\documentclass{amsart}
\usepackage{amsfonts}
\usepackage{amsmath}
 \title{\textbf{Nearly Sasakian geometry and $SU(2)$-structures}}

 \author[B. Cappelletti-Montano]{Beniamino Cappelletti-Montano}
 \address{Dipartimento di Matematica e Informatica, Universit\`a degli Studi di
 Cagliari, Via Ospedale 72, 09124 Cagliari, Italy}
 \email{b.cappellettimontano@gmail.com}

 \author[G. Dileo]{Giulia Dileo}
 \address{Dipartimento di Matematica, Universit\`a degli Studi di
 Bari ``Aldo Moro'', Via E. Orabona 4, 70125 Bari, Italy}
 \email{giulia.dileo@uniba.it}


\thanks{Research partially supported by PRIN 2010/11 -- Variet\`{a} reali e complesse: geometria, topologia e analisi armonica - Italy}


   \date{}
\usepackage{graphicx}
\usepackage{amsmath}
\newtheorem{theorem}{Theorem}[section]

\newtheorem{corollary}[theorem]{Corollary}

\newtheorem{example}[theorem]{Example}

\newtheorem{lemma}[theorem]{Lemma}

\newtheorem{proposition}[theorem]{Proposition}
\newtheorem{remark}[theorem]{Remark}

\begin{document}

\maketitle

\begin{abstract}
We carry on a systematic study of nearly Sasakian manifolds. We prove that any nearly Sasakian
manifold admits two types of integrable distributions with totally geodesic leaves which are, respectively, Sasakian and $5$-dimensional nearly Sasakian
manifolds. As a consequence, any nearly
Sasakian manifold is a contact manifold. Focusing on the $5$-dimensional case, we prove that there
exists a one-to-one correspondence between nearly Sasakian structures and a special class of nearly hypo $SU(2)$-structures. By deforming such an
$SU(2)$-structure one obtains in fact a Sasaki-Einstein structure. Further we prove that both nearly Sasakian and
Sasaki-Einstein $5$-manifolds are endowed with supplementary nearly
cosymplectic structures. We show that there is a one-to-one correspondence between nearly cosymplectic structures and a special class of hypo $SU(2)$-structures which is again strictly related to
Sasaki-Einstein structures. Furthermore, we study the orientable hypersurfaces of a nearly K\"{a}hler
6-manifold and, in the last part of the paper, we define canonical
connections for nearly Sasakian manifolds, which  play a role similar to
the  Gray connection in the context of nearly K\"{a}hler geometry. In dimension $5$ we determine a connection which parallelizes all
the nearly Sasakian $SU(2)$-structure as well as the torsion tensor field. An analogous result holds also for Sasaki-Einstein structures.
\end{abstract}

\section{Introduction}
Nearly K\"{a}hler manifolds were defined by Gray \cite{Gray2} as almost Hermitian manifolds $(M,J,g)$ such that the Levi-Civita connection satisfies \begin{equation*}
(\nabla_{X}J)Y + (\nabla_{Y}J)X = 0
\end{equation*}
for any vector fields $X$ and $Y$ on $M$. The development of nearly K\"{a}hler geometry is mainly due to the studies of Gray \cite{Gray2}, \cite{Gray3}, \cite{Gray} and, more recently, to the  work of Nagy (\cite{Nagy1}, \cite{Nagy2}). Nearly Sasakian manifolds where introduced by Blair, Yano and Showers in \cite{BlairSY} as an odd dimensional counterpart of nearly K\"{a}hler manifolds, together with nearly cosymplectic manifolds, studied by Blair and Showers some years earlier (\cite{BLAIR_cos}, \cite{BLAIR-cos2}). Namely, a smooth manifold $M$ endowed with an almost contact metric structure $(\phi,\xi,\eta,g)$ is said to be nearly Sasakian or nearly cosymplectic if, respectively,
\begin{gather*}
(\nabla_X\phi)Y+(\nabla_Y\phi)X=2g(X,Y)\xi-\eta(X)Y-\eta(Y)X,\\
 (\nabla_X\phi)Y+(\nabla_Y\phi)X=0
\end{gather*}
for every vector fields $X$ and $Y$ on $M$. Since the foundational articles of Blair and his collaborators, these two classes of almost contact structures were studied by some authors and, later on, have played a role in the Chinea-Gonzalez's classification  of almost contact metric manifolds (\cite{Chinea}). Recently, they naturally appeared in the study of harmonic almost contact structures (cf. \cite{Gonzalez}, \cite{Vergara2}, \cite{Vergara1}).

Actually it is more difficult than expected to find relations between nearly Sasakian and nearly K\"{a}hler manifolds, like for Sasakian / K\"{a}hler geometry. For instance, it is known that, like Sasakian manifolds, the Reeb vector field $\xi$ of any nearly Sasakian manifold $M$ defines a Riemannian foliation. Then one would expect that the space of leaves of this foliation is nearly K\"{a}hler, but this happens  if and only if $M$ is Sasakian, and in that case the space of leaves is K\"{a}hler. Moreover, it is not difficult to see that the cone over $M$ is nearly K\"{a}hler if and only if $M$ is Sasakian and, again, in this case the cone is K\"{a}hler. Similar results hold also in the nearly cosymplectic setting. For instance, if one applies the Morimoto's construction \cite{morimoto}  to the product $N$ of two nearly cosymplectic manifolds $M_1$ and $M_2$, one finds that $N$ is nearly K\"{a}hler if and only if both $M_1$ and $M_2$ are coK\"{a}hler.

In the present paper we show in fact that there are many differences between nearly K\"{a}hler and nearly Sasakian manifolds, much more than in K\"{a}hler / Sasakian setting.

It is known that the structure $(1,1)$-tensor field $\phi$ of a Sasakian manifold is given by the opposite of the covariant derivative of the Reeb vector field. Thus in any nearly Sasakian manifold one is lead to define a tensor field $h$ by
\begin{equation*}
\nabla\xi = - \phi + h.
\end{equation*}
This tensor field measures, somehow, the non-Sasakianity of the manifold and plays an important role in our study. Namely, first we prove that the eigenvalues of the symmetric operator $h^2$ are constants and its spectrum is of type
\begin{equation*}
\textrm{Spec}(h^2) = \left\{0, -\lambda_{1}^{2}, \ldots, -\lambda_{r}^{2} \right\}
\end{equation*}
with $\lambda_i \neq 0$ for each $i\in\left\{1,\ldots,r\right\}$.  Then we prove the following theorem.

\begin{theorem}\label{struttura}
Let $M$ be a (non-Sasakian) nearly Sasakian manifold with structure $(\phi,\xi,\eta,g)$. Then the tangent bundle of $M$ splits as the orthogonal sum
\begin{equation*}
TM = {\mathcal D}(0)\oplus{\mathcal D}(-\lambda_{1}^2)\oplus\cdots\oplus{\mathcal D}(-\lambda_{r}^2)
\end{equation*}
of the eigendistributions of $h^2$. Moreover,
\begin{enumerate}
  \item[a)] the distribution ${\mathcal D}(0)$ is integrable and defines a totally geodesic foliation of $M$ of dimension $2p+1$.  If $p>0$ then the leaves of ${\mathcal D}(0)$ are Sasakian manifolds;\medskip
  \item[b)] each distribution $[\xi]\oplus {\mathcal D}(-\lambda_{i}^2)$ is integrable and defines a totally geodesic foliation of $M$ whose leaves are $5$-dimensional nearly Sasakian non-Sasakian manifolds.
\end{enumerate}
Furthermore, if $p>0$ the distribution $[\xi]\oplus{\mathcal D}(-\lambda_{1}^2)\oplus\cdots\oplus{\mathcal D}(-\lambda_{r}^2)$ is integrable and defines a Riemannian foliation with totally geodesic leaves, whose leaf space is K\"{a}hler.
\end{theorem}

As a consequence of Theorem \ref{struttura} we shall prove that in every nearly Sasakian manifold the $1$-form $\eta$ is a contact form. This  establishes a sensible difference with respect to nearly K\"{a}hler geometry, since in any nearly K\"{a}hler manifold the K\"{a}hler form is symplectic if and only if the manifold is K\"{a}hler.

The point b) of Theorem \ref{struttura} motivates us to further investigate $5$-dimensional nearly Sasakian manifolds. Some early studies date back to Olszak (\cite{Olszak1}) who proved that $5$-dimensional nearly Sasakian non-Sasakian manifolds are Einstein and of scalar curvature $>20$.
In the present paper we characterize nearly Sasakian structures in terms of $SU(2)$-structures defined by a $1$-form $\eta$ and a triple $(\omega_{1},\omega_{2},\omega_{3})$ of $2$-forms according to \cite{CS}.
One of our main results  is to prove that there exists a one-to-one correspondence between nearly Sasakian structures on a $5$-manifold and $SU(2)$-structures satisfying the following equations
\begin{equation}\label{nearluSU2}
d\eta=-2\omega_3+2\lambda\omega_1,\qquad d\omega_1=3\eta\wedge\omega_2,\qquad d\omega_2=-3\eta\wedge\omega_1-3\lambda\eta\wedge\omega_3,
\end{equation}
for some real number $\lambda\neq 0$ which depends only on the geometry of the manifold via the formula $s=20(1+\lambda^2)$,
where $s$ is the scalar curvature.  By deforming $(\eta,\omega_1,\omega_2,\omega_3)$ we obtain a Sasaki-Einstein structure with the same underlying contact form (up to a multiplicative factor) and, conversely, each Sasaki-Einstein $5$-manifold carries a nearly Sasakian structure (in fact, a $1$-parameter family of nearly Sasakian structures).

In Section 5 we get analogous results in terms of $SU(2)$-structures for nearly cosymplectic $5$-manifolds. In particular we prove that any nearly cosymplectic $5$-manifold is Einstein with positive scalar curvature. We also show that nearly cosymplectic structures arise naturally both in nearly Sasakian and in Sasaki-Einstein $5$-manifolds. In particular, it is known that any Sasaki-Einstein $SU(2)$-structure can be described by the data of three almost contact metric structures $(\phi_1,\xi,\eta,g)$, $(\phi_2,\xi,\eta,g)$, $(\phi_3,\xi,\eta,g)$, with the same Reeb vector field, satisfying the quaternionc-like relations
\begin{equation*}
\phi_i\phi_j=\phi_k=-\phi_j\phi_i
\end{equation*}
for any even permutation $(i,j,k)$ of $(1,2,3)$ and such that $(\phi_3,\xi,\eta,g)$ is Sasakian with Einstein Riemannian metric $g$. Actually we prove that $(\phi_{1},\xi,\eta,g)$ and $(\phi_{2},\xi,\eta,g)$ are both nearly cosymplectic.

In Section \ref{hypersurfaces} we study the (orientable)  hypersurfaces of a nearly K\"{a}hler $6$-manifolds. In particular we study the $SU(2)$-structures induced on hypersurfaces whose second fundamental form is of type $\sigma = \beta(\eta\otimes\eta)\nu$ or $\sigma=(-g+\beta(\eta\otimes\eta))\nu$, for some function $\beta$, where $\nu$ denotes the unit normal vector field. In both cases we prove that the hypersurface carries a Sasaki-Einstein structure, thus generalizing a result of \cite{FISU}.


Finally, in the last section of the paper, we try to define a canonical connection for nearly Sasakian manifolds, which may play a role similar to the  Gray connection in the context of nearly K\"{a}hler geometry, i.e. the unique Hermitian connection with totally skew-symmetric torsion. In \cite{FrIv} Friedrich and Ivanov provided necessary and sufficient conditions for an almost contact metric manifold to admit a (unique) connection with totally skew-symmetric torsion parallelizing all the structure tensors. One can easily deduce that a nearly Sasakian manifold admits such a connection if and only if it is Sasakian. Thus, weakening some hypotheses, we define a family of connections, parameterized by a real number $r$, which  parallelize the almost contact metric structure and such that the torsion is skew-symmetric on the contact distribution $\ker(\eta)$. In particular, if $M$ is a Sasakian manifold  our connection  coincides with the Okumura connection \cite{Okumura}. In dimension $5$ the connection corresponding to the value $r = \frac{1}{2}$ parallelizes all the tensors in the associated $SU(2)$-structure $(\eta,\omega_1,\omega_2,\omega_3)$, as well as the torsion tensor field. Then for Sasaki-Einstein $5$-manifolds we prove that the Okumura connection corresponding to $r=\frac{1}{2}$ parallelizes the whole $SU(2)$-structure.

\medskip

All manifolds considered in this paper will be assumed to be smooth i.e. of the class $C^{\infty}$, and connected. We use the
convention that $u\wedge v = u \otimes v - v \otimes u$. Unless in the last Section, we shall implicitly assume that all the nearly Sasakian  (respectively, nearly cosymplectic) manifolds considered in the paper are non-Sasakian (respectively, non-coK\"{a}hler).

\section{Preliminaries}
An almost contact metric manifold is a differentiable manifold $M^{2n+1}$ endowed with a
structure $(\phi, \xi, \eta, g)$, given by a tensor field $\phi$ of type $(1,1)$, a
vector field $\xi$, a $1$-form $\eta$ and a Riemannian metric
$g$ satisfying
\[\phi^2={}-I+\eta\otimes\xi,\quad \eta(\xi)=1,\quad g(\phi X,\phi Y)=g(X,Y)-\eta(X)\eta(Y)\]
for every vector fields $X,Y$ on $M$. From the definition it follows that $\phi\xi=0$ and $\eta\circ\phi=0$. Moreover one has that $g(X,\phi Y)=-g(\phi X, Y)$ so that the bilinear form $\Phi:=g(-,\phi-)$ defines in fact a $2$-form on $M$, called \emph{fundamental $2$-form}.

Two remarkable classes of almost contact metric manifolds are given by Sasakian and coK\"{a}hler manifolds. An almost contact metric manifold is said to be \emph{Sasakian} if tensor field $N_{\phi}:=[\phi,\phi]+\eta\otimes\xi$ vanishes identically and $d\eta=2\Phi$, \emph{coK\"{a}hler} if $N_{\phi}\equiv 0$ and $d\eta=0$, $d\Phi=0$. The Sasakian and coK\"{a}hler conditions can be equivalently expressed in terms of the Levi-Civita connection by, respectively,
\begin{gather*}
(\nabla_{X}\phi)Y=g(X,Y)\xi-\eta(Y)X,\\
\nabla \phi =0.
\end{gather*}
For further details on Sasakian and coK\"{a}hler manifolds we refer to \cite{BLAIR,boyergalicki2008} and \cite{CappDenYud}, respectively.

An almost contact metric manifold $(M,\phi, \xi,\eta,g)$ is called \emph{nearly Sasakian} if the covariant derivative of $\phi$ with respect to the Levi-Civita connection $\nabla$ satisfies
\begin{equation}\label{main}
(\nabla_X\phi)Y+(\nabla_Y\phi)X=2g(X,Y)\xi-\eta(X)Y-\eta(Y)X
\end{equation}
for every vector fields $X,Y$ on $M$, or equivalently,
\[(\nabla_X\phi)X=g(X,X)\xi-\eta(X)X\]
for every vector field $X$ on $M$. This notion was introduced in \cite{BlairSY} in order to study an odd dimensional counterpart of nearly K\"{a}hler geometry, and then it was studied by other authors.
One can easily check that \eqref{main} is also equivalent to
\begin{equation}\label{dPhi}
3g((\nabla_X\phi)Y,Z)=-d\Phi(X,Y,Z)-3\eta(Y)g(X,Z)+3\eta(Z)g(X,Y).
\end{equation}

We recall now some basic properties satisfied by nearly Sasakian structures which will be used in the following. We refer to \cite{BlairSY, Olszak, Olszak1} for the details.

It is known that the characteristic vector field $\xi$ is Killing and the Levi-Civita connection satisfies $\nabla_\xi\xi=0$ and $\nabla_\xi\eta=0$.
One can define a tensor field $h$ of type $(1,1)$ by putting
\begin{equation}\label{nablaxi}
\nabla_X\xi=-\phi X+hX.
\end{equation}
The operator $h$ is skew-symmetric and anticommutes with $\phi$. Moreover, $h\xi=0$  and $\eta\circ h=0$. The vanishing of $h$ provides a necessary and sufficient condition for a nearly Sasakian manifold to be Sasakian (\cite{Olszak1}). Applying \eqref{main} and \eqref{nablaxi}, one easily gets
\begin{equation}\label{nablaxiphi}
\nabla_{\xi}\phi=\phi h.
\end{equation}
We remark the circumstance that the operator $h$ is also related to the Lie derivative of $\phi$ with respect to $\xi$. Indeed, using \eqref{nablaxiphi} and \eqref{nablaxi}, we get
\[(\mathcal{L}_\xi\phi )X=[\xi,\phi X]-\phi[\xi, X]  =(\nabla_\xi\phi)X-\nabla_{\phi X}\xi+\phi(\nabla_X\xi)   = 3\phi hX. \]
Denote by $R$ the Riemannian curvature tensor. Olszak proved the following formula in \cite{Olszak}:
\begin{equation}\label{Rxi}
R(\xi,X)Y=(\nabla_X\phi)Y-(\nabla_Xh)Y=g(X-h^2X,Y)\xi-\eta(Y)(X-h^2X).
\end{equation}
The above equation, together with \eqref{nablaxiphi}, gives
\begin{equation}\label{nablaxih}
\nabla_\xi h=\nabla_\xi \phi=\phi h.
\end{equation}
Furthermore, taking $Y=\xi$ in \eqref{Rxi}, we obtain
\[R(X,\xi)\xi=-\eta(X)\xi+X-h^2X=-\phi^2X-h^2X\]
and the $\xi$-sectional curvatures for every unit vector field $X$ orthogonal to $\xi$ are
\[K(\xi,X)=g(R(X,\xi)\xi,X)=1+g(hX,hX)\geq 1.\]
Notice that \eqref{Rxi} also implies that
\begin{equation}\label{RXYxi}
R(X,Y)\xi=\eta(Y)X-\eta(X)Y-\eta(Y)h^2X+\eta(X)h^2Y.
\end{equation}
Moreover, the Ricci curvature satisfies
\begin{equation*}\label{Ricci3}
\textrm{Ric}(\phi X,\phi Y)=\textrm{Ric}(X,Y)-(2n-\mathrm{tr} (h^2))\eta(X)\eta(Y).
\end{equation*}
In particular it follows that the Ricci operator commutes with $\phi$. Finally, Olszak proved that the symmetric operator $h^2$ has constant trace and the covariant derivatives of $\phi$ and $h^2$ satisfy the following relations:
\begin{equation}\label{varie5}
g((\nabla_X\phi)Y, hZ)=\eta(Y)g(h^2X,\phi Z)-\eta(X)g(h^2Y,\phi Z)+\eta(Y)g(hX,Z),
\end{equation}
\begin{equation}\label{varie7}
(\nabla_Xh^2)Y=\eta(Y)(\phi -h)h^2X+g((\phi -h)h^2X,Y)\xi.
\end{equation}
\medskip

We now recall some facts about nearly cosymplectic manifolds. A \emph{nearly cosymplectic manifold} is an almost contact metric manifold $(M,\phi, \xi,\eta,g)$ such that the covariant derivative of $\phi$ with respect to the Levi-Civita connection $\nabla$ satisfies
\begin{equation}\label{main_c}
(\nabla_X\phi)Y+(\nabla_Y\phi)X=0
\end{equation}
for every vector fields $X,Y$.  The above condition is equivalent to $(\nabla_X\phi)X=0$, or also to
\begin{equation}\label{dPhi_c}
3g((\nabla_X\phi)Y,Z)=-d\Phi(X,Y,Z)
\end{equation}
for any $X,Y,Z\in \mathfrak{X}(M)$.
Also in this case we have that $\xi$ is Killing, $\nabla_\xi\xi=0$ and $\nabla_\xi\eta=0$.
The tensor field $h$ of type $(1,1)$ defined by
\begin{equation}\label{nablaxi_c}
\nabla_X\xi=hX
\end{equation}
is skew-symmetric and anticommutes with $\phi$. It satisfies $h\xi=0$, $\eta\circ h=0$ and
\begin{equation}\label{nearlycos-nablaxiphi}
\nabla_\xi\phi=\phi h.
\end{equation}
Furthermore, $h$ is related to the Lie derivative of $\phi$ in the direction of $\xi$. Indeed,
\[({\mathcal L}_\xi\phi)X=(\nabla_\xi\phi)X-\nabla_{\phi X}\xi+\phi(\nabla_X\xi)=3\phi hX.\]
Finally, the following formulas hold (\cite{E}):
\begin{align}
g((\nabla_X\phi)Y, hZ)&=\eta(Y)g(h^2X,\phi Z)-\eta(X)g(h^2Y,\phi Z),\label{nablaphi_hc}\\
(\nabla_Xh)Y&=g(h^2X,Y)\xi-\eta(Y)h^2X,\label{nablah_c}\\
\mathrm{tr}(h^2)&=\mathrm{constant}.\label{tr}
\end{align}

\section{The foliated structure of a nearly Sasakian manifold}\label{foliationsection}

In this section we show that any nearly Sasakian manifold is foliated by two types of foliations, whose leaves are respectively Sasakian or $5$-dimensional nearly Sasakian non-Sasakian manifolds. An important role in this context is played by the symmetric operator $h^2$ and  by its spectrum $\textrm{Spec}(h^2)$. We recall the following result.

\begin{theorem}[\cite{Olszak}]\label{condizione-olszak}
If a nearly Sasakian manifold $M$ satisfies the condition
\begin{equation*}
h^2 = \lambda (I-\eta\otimes\xi)
\end{equation*}
for some real number $\lambda$, then $\dim(M)=5$.
\end{theorem}

\begin{proposition}
The eigenvalues of the operator $h^2$ are constant.
\end{proposition}
\begin{proof}
Let $\mu$ be an eigenvalue of $h^2$ and let $Y$ be a local unit vector field orthogonal to $\xi$ such that $h^2Y=\mu Y$. Applying \eqref{varie7} for any vector field $X$, and taking $Y=Z$
we get
\begin{align*}
0&=g((\nabla_Xh^2)Y,Y)\\
&=g(\nabla_X(h^2Y),Y)-g(h^2(\nabla_XY),Y)\\
&=X(\mu)g(Y,Y)+\mu g(\nabla_XY,Y)-g(\nabla_XY,h^2Y)\\
&=X(\mu)g(Y,Y)
\end{align*}
which implies that $X(\mu)=0$.
\end{proof}
\medskip

Notice that $0$ is an eigenvalue of $h^2$, since $h\xi=0$. Furthermore, being $h$ skew-symmetric,
the non-vanishing eigenvalues of $h^2$ are negative, so that the spectrum of $h^2$ is of type
\[\textrm{Spec}(h^2)=\{0,-\lambda_1^2,\ldots,-\lambda_r^2\},\]
$\lambda_i\ne0$ and $\lambda_i\neq\lambda_j$ for $i\ne j$.
Further, if $X$ is an eigenvector of $h^2$ with eigenvalue $-\lambda_i^2$, then $X$, $\phi X$, $hX$, $h\phi X$ are orthogonal eigenvectors of $h^2$ with eigenvalue $-\lambda_i^2$.

In the following we denote by $[\xi]$ the $1$-dimensional distribution generated by $\xi$,
and by ${\mathcal D}(0)$ and ${\mathcal D}(-\lambda_i^2)$ the distributions of the eigenvectors $0$ and
$-\lambda_i^2$ respectively.

\begin{theorem}\label{main1}
Let $M$ be a nearly Sasakian manifold with structure $(\phi,\xi,\eta,g)$ and let
$\mathrm{Spec}(h^2)=\{0,-\lambda_1^2,\ldots,-\lambda_r^2\}$ be the spectrum of $h^2$.
Then the distributions $\mathcal D(0)$ and $[\xi]\oplus\mathcal D(-\lambda_i^2)$
are integrable with totally geodesic leaves. In particular,
\begin{itemize}
\item[\textrm{a)}] the eigenvalue $0$ has multiplicity $2p+1$, $p\geq0$. If $p>0$, the leaves of $\mathcal D(0)$ are
$(2p+1)$-dimensional Sasakian manifolds;
\item[\textrm{b)}] each negative eigenvalue $-\lambda_i^2$ has multiplicity $4$ and the leaves of the
distribution $[\xi]\oplus {\mathcal D}(-\lambda_i^2)$ are $5$-dimensional nearly Sasakian (non-Sasakian) manifolds.
\end{itemize}
Therefore, the dimension of $M$ is $1+2p+4r$.
\end{theorem}
\begin{proof}
Consider an eigenvector $X$ with eigenvalue $\mu$. From \eqref{nablaxi} we deduce that
$\nabla_X\xi$ is an eigenvector with eigenvalue $\mu$.
On the other hand, \eqref{varie7} implies $\nabla_\xi h^2=0$,
so that $\nabla_\xi X$ is also an eigenvector with eigenvalue $\mu$.

Now, if $X,Y$ are eigenvectors with eigenvalue $\mu$,
orthogonal to $\xi$, from \eqref{varie7}, we get
\[h^2(\nabla_XY)=\mu\nabla_XY-\mu g(\phi X-hX,Y)\xi.\]
If $\mu=0$, we immediately get that $\nabla_XY\in {\mathcal D}(0)$. If $\mu\ne 0$, we have
\[h^2(\phi^2\nabla_XY)=\phi^2(h^2\nabla_XY)=\mu \phi^2(\nabla_XY)\]
and thus $\nabla_XY=-\phi^2 \nabla_XY +\eta(\nabla_XY)\xi$ belongs to the distribution
$[\xi]\oplus {\mathcal D}(\mu)$. This proves the first part of the Theorem.

If $X$ is an eingenvector of $h^2$ orthogonal to $\xi$, with eigenvalue $\mu$, also
 $\phi X$ is an eingenvector with the same eigenvalue $\mu$. Hence, the eigenvalue $0$ has odd multiplicity $2p+1$, for some integer $p\geq 0$. If $p>0$ the structure $(\phi,\xi,\eta,g)$ induces a nearly Sasakian structure on the leaves of the distribution $\mathcal D(0)$ whose associated tensor $h$ vanishes. Therefore, the induced structure is Sasakian.

As regards b), since $\phi$ preserves each distribution ${\mathcal D}(-\lambda_{i}^{2})$, the structure $(\phi,\xi,\eta,g)$ induces a nearly Sasakian structure on the leaves of the distribution $[\xi]\oplus\mathcal D(-\lambda_i^2)$, which we denote in the same manner. For such a structure the operator $h$ satisfies
\[h^2=-\lambda_i^2(I-\eta\otimes\xi).\]
By Theorem \ref{condizione-olszak}, the leaves of this distribution are $5$-dimensional, so that the multiplicity of the eigenvalue $-\lambda_i^2$ is $4$.
\end{proof}

Now using Theorem \ref{main1} we prove that every nearly Sasakian manifold is foliated by another foliation, which is both Riemannian and totally geodesic, and such that the leaf space is K\"{a}hler. Before we need the following preliminary result.

\begin{lemma}\label{utile}
Let $(M,\phi,\xi,\eta,g)$ be a nearly Sasakian manifold. For any $X\in{\mathcal D}(-\lambda_{i}^{2})$, $i\in\left\{1,\ldots,r\right\}$, and for any $Z\in{\mathcal D}(0)$ one has that $\nabla_{Z}X\in {\mathcal D}(-\lambda_{1}^{2})\oplus\cdots\oplus{\mathcal D}(-\lambda_{r}^{2})\oplus[\xi] $.
\end{lemma}
\begin{proof}
For any $Z'\in{\mathcal D}(0)$ orthogonal to $\xi$, since the distribution ${\mathcal D}(0)$ is integrable with totally geodesic leaves, we have that
$g\left(\nabla_{Z}X , Z'\right) = - g\left( \nabla_{Z}Z' , X\right)=0.$
\end{proof}

\begin{theorem}\label{main6}
With the notation  of Theorem \ref{main1}, assuming $p>0$, the distribution $\mathcal{D}(-\lambda_{1}^2)\oplus\cdots\oplus\mathcal{D}(-\lambda_{r}^2)\oplus\left[\xi\right]$ is integrable and defines a transversely K\"{a}hler foliation with totally geodesic leaves.
\end{theorem}
\begin{proof} We already know that each distribution $[\xi]\oplus\mathcal{D}(-\lambda_i^2)$ is integrable with totally geodesic leaves. Moreover, by \eqref{varie7}, one has for any $X\in{\mathcal D}(-\lambda_{i}^2)$, $Y\in{\mathcal D}(-\lambda_{j}^2)$ and $Z\in{\mathcal D}(0)$  orthogonal to $\xi$,
\[
g(\nabla_{X}Y,Z)=-\frac{1}{\lambda_{j}^{2}}g(\nabla_{X} h^{2}Y,Z)=-\frac{1}{\lambda_{j}^{2}}g( (\nabla_{X}h^2)Y + h^{2}\nabla_{X}Y , Z )=-\frac{1}{\lambda_{j}^{2}}  g(\nabla_{X}Y, h^{2}Z)=0.
\]
Now we prove that $\mathcal{D}(-\lambda_{1}^2)\oplus\cdots\oplus\mathcal{D}(-\lambda_{r}^2)\oplus\left[\xi\right]$ defines a Riemannian foliation. First, for any $Z,Z'\in{\mathcal D}(0)$, $({\mathcal L}_{\xi}g)(Z,Z')=0$ since $\xi$ is Killing.
Next, by applying Lemma \ref{utile} we conclude that, for any $X \in {\mathcal D}(-\lambda_{i}^2)$,
\begin{equation*}
({\mathcal L}_{X}g)(Z,Z') =g(\nabla_{Z}X,Z') + g(\nabla_{Z'}X,Z)=0.
\end{equation*}
Now let us prove that also the tensor field  $\phi$ is projectable, i.e. it maps basic vector fields into basic vector fields. Let $Z \in {\mathcal D}(0)$, $Z$ orthogonal to $\xi$, be a basic vector field, that is $[\xi,Z], [X,Z] \in {\mathcal D}(-\lambda_{1}^{2})\oplus\cdots\oplus{\mathcal D}(-\lambda_{r}^{2})\oplus[\xi]$  for any $X\in{\mathcal D}(-\lambda_{i}^{2})$.  Let us prove that $g([X,\phi Z], Z')=0$ for any $Z'\in{\mathcal D}(0)$ orthogonal to $\xi$. By using \eqref{dPhi} and Lemma \ref{utile} we get
\begin{align}\label{intermedia2}
g\left( [X,\phi Z], Z'\right) &=g\left(\nabla_{X}\phi Z,Z'\right) - g\left(\nabla_{\phi Z}X,Z'\right) \nonumber \\
&= g\left( (\nabla_{X}\phi)Z,Z'\right) + g\left(\phi\nabla_{X}Z,Z'\right) \nonumber\\
&=-\frac{1}{3}d\Phi(X,Z,Z')- g(\nabla_{X}Z,\phi Z').
\end{align}
Let us check that each summand in  \eqref{intermedia2} vanishes. First notice that, since $Z$ is basic and again by Lemma \ref{utile}, one has $g(\nabla_{X}Z,\phi Z')=g(\nabla_{Z}X, \phi Z') + g([X,Z],\phi Z')=0$.
Further, since the Riemannian metric $g$ is bundle-like, and using Lemma \ref{utile} and \eqref{dPhi}, we have
\begin{align*}
 d \Phi(X,Z,Z') &= X(\Phi(Z,Z')) - \Phi([X,Z],Z') - \Phi([Z,Z'],X) - \Phi([Z',X],Z)\\
&=X(g(Z,\phi Z')) - g([X,Z],\phi Z') - g([Z',X],\phi Z)\\
&=g([X, \phi Z'],Z)  - g([Z',X],\phi Z)\\
&=g(\nabla_{X}\phi Z', Z) - g(\nabla_{\phi Z'}X,Z) -g(\nabla_{Z'}X,\phi Z) + g(\nabla_{X}Z',\phi Z)\\
&=g\left(\left(\nabla_{X}\phi\right)Z',Z\right)\\
&=\frac{1}{3}d\Phi(X,Z,Z'),
\end{align*}
from which it follows that $d \Phi(X,Z,Z')=0$. Therefore, in view of \eqref{intermedia2}, we have $g\left( [X,\phi Z], Z'\right)=0$ for any $Z'\in{\mathcal D}(0)$ orthogonal to $\xi$, and thus we conclude that $\phi Z$ is basic.

Thus we have proved that the Riemannian metric $g$ and the tensor field $\phi$ are projectable with respect to the foliation ${\mathcal D}(-\lambda_{1}^{2})\oplus\cdots\oplus{\mathcal D}(-\lambda_{r}^{2})\oplus[\xi]$. Finally, from \eqref{Rxi}, the integrability of ${\mathcal D}(0)$ and $h\xi=0$ it follows that ${\mathcal D}(-\lambda_{1}^{2})\oplus\cdots\oplus{\mathcal D}(-\lambda_{r}^{2})\oplus[\xi]$ is transversely K\"{a}hler.
\end{proof}

In view of Theorem \ref{main1} it becomes of great importance the study of $5$-dimensional nearly Sasakian  manifolds. This will be precisely the subject of the next Section.

\section{Nearly Sasakian manifolds and $SU(2)$-structures}\label{SU2section}

Let $M$ be a $5$-dimensional manifold. An $SU(2)$-structure on $M$, that is an $SU(2)$-reduction of the bundle $L(M)$ of linear frames on $M$, is equivalent to the existence of three almost contact metric structures $(\phi_1,\xi,\eta,g)$, $(\phi_2,\xi,\eta,g)$, $(\phi_2,\xi,\eta,g)$ related by
\begin{equation}\label{quaternionic}
\phi_i\phi_j=\phi_k=-\phi_j\phi_i
\end{equation}
for any even permutation $(i,j,k)$ of $(1,2,3)$. In \cite{CS} Conti and Salamon proved that, in the spirit of special geometries, such a structure is equivalently determined by a quadruplet $(\eta,\omega_1,\omega_2,\omega_3)$, where $\eta$ is a $1$-form and $\omega_i$, $i\in\{1,2,3\}$, are $2$-forms, satisfying
\begin{equation}\label{SU2}
\omega_i\wedge\omega_j=\delta_{ij}v
\end{equation}
for some $4$-form $v$ with $v\wedge\eta\ne0$, and
\begin{equation}\label{SU2b}
X\lrcorner\,\omega_1=Y\lrcorner\,\omega_2\Longrightarrow\omega_3(X,Y)\geq 0.
\end{equation}
The endomorphisms $\phi_i$ of $TM$, the Riemannian metric $g$ and the $2$-forms $\omega_i$ are related by
\[\omega_i(X,Y)=g(\phi_i X,Y),\]
(see also \cite{BV}). A well-known class of $SU(2)$-structures on a $5$-dimensional manifold is given by \emph{Sasaki-Einstein structures}, characterized by the following differential equations:
\begin{equation}\label{S-E}
d\eta=-2\omega_3,\qquad d\omega_1=3\eta\wedge\omega_2,\qquad d\omega_2=-3\eta\wedge\omega_1.
\end{equation}
For such a manifold the almost contact metric structure $(\phi_3,\xi,\eta,g)$ is Sasakian, with Einstein Riemannian metric $g$. A Sasaki-Einstein $5$-manifold may be equivalently defined as a Riemannian manifold $(M,g)$ such that the product $M\times \mathbb{R}_+$ with the cone metric $dt^2+t^2g$ is K\"ahler and Ricci-flat (Calabi-Yau).

In \cite{CS}, Conti and Salamon introduced hypo structures as a natural generalization of Sasaki-Einstein structures. Indeed, an $SU(2)$-structure $(\eta,\omega_1,\omega_2,\omega_3)$ is called \emph{hypo} if
\begin{equation}\label{hypo}
d\omega_3=0,\qquad d(\eta\wedge\omega_1)=0,\qquad d(\eta\wedge\omega_2)=0.
\end{equation}
These structures arise naturally on hypersurfaces of $6$-manifolds endowed with an integrable $SU(3)$-structure.
In \cite{FISU} the authors introduced \emph{nearly hypo} structures, defined as $SU(2)$-structures $(\eta,\omega_1,\omega_2,\omega_3)$ satisfying
\begin{equation}
d\omega_1=3\eta\wedge\omega_2,\qquad d(\eta\wedge\omega_3)=-2\omega_1\wedge\omega_1.
\end{equation}
Such  structures arise on hypersurfaces of nearly K\"ahler $SU(3)$-manifolds.

We shall provide an equivalent notion of nearly Sasakian $5$-manifolds in terms of $SU(2)$-structures. First we state the following lemmas.

\begin{lemma}\label{SU2-nablaphi}
Let $M$ be a $5$-manifold with an $SU(2)$-structure $\left\{(\phi_i,\xi,\eta,g)\right\}_{i\in\left\{1,2,3\right\}}$. Then for any even permutation $(i,j,k)$ of $(1,2,3)$, we have
\begin{align}\label{N}
g(N_{\phi_i}(X,Y),\phi_j Z)&={}-d\omega_j(X,Y,Z)+d\omega_j(\phi_i X,\phi_i Y,Z)\\&\quad\,\,{}+d\omega_k(\phi_i X,Y,Z)+d\omega_k(X,\phi_i Y,Z).\nonumber
\end{align}
\end{lemma}
\begin{proof}
A simple computation using the quaternionic identities \eqref{quaternionic} shows that
\begin{equation*}
\phi_{i}(\nabla_Z\phi_{j})\phi_{i}=-\phi_{i}\nabla_Z\phi_{k}-\nabla_Z\phi_{j}+(\nabla_Z\phi_{k})\phi_{i}.
\end{equation*}
Therefore
\begin{align}\label{nablaomega}
(\nabla_Z\omega_{j})(\phi_{i}X,\phi_{i}Y)&=-g(\phi_{i}(\nabla_Z\phi_{j})\phi_{i}X,Y)\nonumber\\
&=-(\nabla_Z\omega_{k})(X,\phi_{i}Y)+(\nabla_Z\omega_{j})(X,Y)-(\nabla_Z\omega_{k})(\phi_{i}X,Y).
\end{align}
The tensor field $N_{\phi_{i}}$ can be written as
\begin{align*}
N_{\phi_{i}}(X,Y)&=(\nabla_{\phi_{i}X}\phi_{i})Y-(\nabla_{\phi_{i}Y}\phi_{i})X+(\nabla_X\phi_{i})\phi_{i}Y-(\nabla_Y\phi_{i})\phi_{i}X+\eta(X)\nabla_Y\xi-\eta(Y)\nabla_X\xi\\
&=(\phi_{i}(\nabla_Y\phi_{i})-\nabla_{\phi_{i}Y}\phi_{i})X-(\phi_{i}(\nabla_X\phi_{i})-\nabla_{\phi_{i}X}\phi_{i})Y+((\nabla_X\eta)(Y)-(\nabla_Y\eta)(X))\xi.
\end{align*}
It follows that
\[\phi_{j}N_{\phi_{i}}(X,Y)=-\phi_{k}(\nabla_Y\phi_{i})X-\phi_{j}(\nabla_{\phi_{i}Y}\phi_{i})X+\phi_{k}(\nabla_X\phi_{i})Y-\phi_{j}(\nabla_{\phi_{i}X}\phi_{i})Y.\]
Now, a straightforward computation shows that
\begin{align*}
g(N_{\phi_{i}}(X,Y),\phi_{j}Z)&=-d\omega_{j}(X,Y,Z)+d\omega_{j}(\phi_{i}X,\phi_{i}Y,Z)+d\omega_{k}(X,\phi_{i}Y,Z)+d\omega_{k}(\phi_{i}X,Y,Z)\\
&\quad\,\,+(\nabla_Z\omega_{j})(X,Y)-(\nabla_Z\omega_{k})(X,\phi_{i}Y)-(\nabla_Z\omega_{k})(\phi_{i}X,Y)\\
&\quad-(\nabla_Z\omega_{j})(\phi_{i}X,\phi_{i}Y).
\end{align*}
Applying \eqref{nablaomega}, we get \eqref{N}.
\end{proof}

\begin{lemma}
Let $M$ be a $5$-manifold endowed with an $SU(2)$-structure $\left\{(\phi_i,\xi,\eta,g)\right\}_{i\in\left\{1,2,3\right\}}$. Then for any even permutation $(i,j,k)$ of $(1,2,3)$ we have
\begin{align}\label{nablaphi}
2g((\nabla_X\phi_{i})Y,Z)&={}-d\omega_{i}(X,\phi_{i} Y,\phi_{i} Z)+d\omega_{i}(X,Y,Z)-d\omega_{j}(Y,Z,\phi_{k}X) \nonumber \\
&\quad+d\omega_{j}(\phi_{i}Y,\phi_{i}Z,\phi_{k}X)+d\omega_{k}(Y,\phi_{i}Z,\phi_{k}X)+d\omega_{k}(\phi_{i}Y,Z,\phi_{k}X) \nonumber \\
&\quad+d\eta(\phi_{i}Y,Z)\eta(X)-d\eta(\phi_{i}Z,Y)\eta(X)+d\eta(\phi_{i}Y,X)\eta(Z)-d\eta(\phi_{i}Z,X)\eta(Y).
\end{align}
\end{lemma}
\begin{proof}
The covariant derivative of $\phi_i$ is given by (see \cite[Lemma 6.1]{BLAIR}):
\begin{align}\label{nablaphi0}
2g((\nabla_X\phi_{i})Y,Z)&={}-d\omega_{i}(X,\phi_{i} Y,\phi_{i} Z)+d\omega_{i}(X,Y,Z)+g(N_{\phi_{i}}(Y,Z),\phi_{i}X)\nonumber\\
&\quad\,\,{}+d\eta(\phi_{i}Y,Z)\eta(X)-d\eta(\phi_{i}Z,Y)\eta(X)+d\eta(\phi_{i}Y,X)\eta(Z)-d\eta(\phi_{i}Z,X)\eta(Y).
\end{align}
Applying \eqref{N} to vector fields $Y,Z$ and $\phi_k X$, being $\phi_j\phi_k=\phi_i$, we have
\begin{align}\label{N1}
g(N_{\phi_{i}}(Y,Z),\phi_{i} X)&={}-d\omega_{j}(Y,Z,\phi_{k}X)+d\omega_{j}(\phi_{i}Y,\phi_{i} Z,\phi_{k}X)\\&\quad\,\,{}+d\omega_{k}(Y,\phi_{i}Z,\phi_{k}X)+d\omega_{k}(\phi_{i}Y,Z,\phi_{k}X).\nonumber
\end{align}
Combining \eqref{nablaphi0} and \eqref{N1}, we get the result.
\end{proof}

\begin{theorem}\label{main2}
Nearly Sasakian  structures on a $5$-dimensional manifold are in one-to-one correspondence with $SU(2)$-structures $(\eta,\omega_1,\omega_2,\omega_3)$ satisfying
\begin{equation}\label{mainSU}
d\eta=-2\omega_3+2\lambda\omega_1,\qquad d\omega_1=3\eta\wedge\omega_2,\qquad d\omega_2=-3\eta\wedge\omega_1-3\lambda\eta\wedge\omega_3
\end{equation}
for some real number $\lambda\ne0$. These $SU(2)$-structures are nearly hypo.
\end{theorem}
\begin{proof}
Let $(M,\phi,\xi,\eta,g)$ be a nearly Sasakian $5$-manifold. The associated tensor $h$ satisfies
\begin{equation}\label{h-5dim}
h^2=-\lambda^2(I-\eta\otimes\xi),
\end{equation}
for some non-vanishing constant $\lambda$. Since $h$ is skew-symmetric, anticommutes with $\phi$ and satisfies $h\xi=0$, the structure tensors $\xi$, $\eta$ and $g$, together with the $(1,1)$-tensor fields
\begin{equation}\label{quaternionic1}
\phi_1:=\frac{1}{\lambda}h,\qquad \phi_2:=\frac{1}{\lambda}\phi h,\qquad \phi_3:=\phi,
\end{equation}
determine an $SU(2)$-reduction of the frame bundle over $M$. Taking the $2$-forms $\omega_i$, $i\in\{1,2,3\}$, defined by
$\omega_i(X,Y):=g(\phi_i X,Y)$,
we prove that the structure $(\eta,\omega_1,\omega_2,\omega_3)$ satisfies \eqref{mainSU}. Using \eqref{nablaxi}, we compute
\begin{align*}
d\eta(X,Y)&= X(\eta(Y))-Y(\eta(X))-\eta([X,Y])\\
&=g(Y,\nabla_X\xi)-g(X,\nabla_Y\xi)\\
&=2 g(-\phi X+hX,Y)\\
&= -2\omega_3(X,Y)+2\lambda\omega_1(X,Y),
\end{align*}
which proves the first equation in \eqref{mainSU}. In particular we have $d\omega_3=\lambda d\omega_1$. Now, by \eqref{dPhi}
\[
d\omega_3(X,Y,Z)= 3g((\nabla_X\phi)Y,Z)+3\eta(Y)g(X,Z)-3\eta(Z)g(X,Y).
\]
For $X=\xi$, applying \eqref{nablaxiphi}, we get
\[d\omega_3(\xi,Y,Z)=3g((\nabla_\xi\phi)Y,Z)=3g(\phi hY,Z)=3\lambda\omega_2(Y,Z).\]
On the other hand, equation \eqref{varie5} implies that for every vector fields $X,Y,Z$ orthogonal to $\xi$, $g((\nabla_X\phi)Y,Z)=0$ and thus $d\omega_3(X,Y,Z)=0$. Therefore
$d\omega_3=3\lambda\eta\wedge\omega_2$. Being also $d\omega_3=\lambda d\omega_1$, we obtain the second equation in \eqref{mainSU}. \ Now, using the first two equations in \eqref{mainSU}, and \eqref{SU2}, we have $\eta\wedge d\omega_2=d\eta\wedge \omega_2=0$,
and thus, for every vector fields $X,Y,Z$ orthogonal to $\xi$,
\[d\omega_2(X,Y,Z)=(\eta\wedge d\omega_2)(\xi,X,Y,Z)=0.\]
From \eqref{nablaxih}, we get $\nabla_\xi(\phi h)=-\lambda^2\phi-h$.
Hence, for every vector fields $Y$, $Z$, using also \eqref{nablaxi}, we compute
\begin{align*}
\lambda d\omega_2(\xi,Y,Z)&=g((\nabla_\xi\phi h)Y,Z)+g((\nabla_Y\phi h)Z,\xi)+g((\nabla_Z\phi h)\xi,Y)\\
&=-3g(hY+\lambda^2\phi Y,Z)\\
&=-3\lambda \omega_1(Y,Z)-3\lambda^2\omega_3(Y,Z),
\end{align*}
and this completes the proof of the third equation in \eqref{mainSU}.

As for the converse, assume that $M$ is a $5$-manifold with an $SU(2)$-structure satisfying \eqref{mainSU} for some non-vanishing real number $\lambda$. Consider the associated almost contact metric structures $(\phi_i,\xi,\eta,g)$, $i\in\{1,2,3\}$.
Applying \eqref{nablaphi} and \eqref{mainSU} we compute the covariant derivative of $\phi_3$:
\begin{align*}
2g((\nabla_X\phi_3)Y,Z)&=-3\lambda\eta(X)\omega_2(\phi_3 Y,\phi_3 Z)+3\lambda\eta(X)\omega_2(Y,Z)
+3\lambda\eta(Y)\omega_2(Z,X)  \\
&\quad+3\lambda\eta(Z)\omega_2(X,Y)-3\eta(Y)\omega_2(Z,\phi_2X)-3\eta(Z)\omega_2(\phi_2X,Y)  \\
&\quad-3\eta(Y)\omega_1(\phi_3Z,\phi_2X)-3\lambda\eta(Y)\omega_3(\phi_3Z,\phi_2X)-3\eta(Z)\omega_1(\phi_2X,\phi_3Y)  \\
&\quad-3\lambda\eta(Z)\omega_3(\phi_2X,\phi_3Y)-2\omega_3(\phi_3Y,Z)\eta(X)+2\lambda\omega_1(\phi_3Y,Z)\eta(X)  \\
&\quad+2\omega_3(\phi_3Z,Y)\eta(X)-2\lambda\omega_1(\phi_3Z,Y)\eta(X)-2\omega_3(\phi_3Y,X)\eta(Z)  \\
&\quad+2\lambda\omega_1(\phi_3Y,X)\eta(Z)+2\omega_3(\phi_3Z,X)\eta(Y)-2\lambda\omega_1(\phi_3Z,X)\eta(Y)  \\
&=\eta(X)\{3\lambda g(\phi_2 Y,Z)+3\lambda g(\phi_2 Y,Z)-2g(\phi_3^2Y,Z)-2\lambda g(\phi_2Y,Z)+2g(\phi_3^2Z,Y)  \\
&\quad+2\lambda g(\phi_2Z,Y)\}+\eta(Y)\{3\lambda g(\phi_2 Z,X)-3 g(\phi_2Z,\phi_2X)+3 g(\phi_2Z,\phi_2X)  \\
&\quad +3\lambda g(Z,\phi_2 X)+2 g(\phi_3^2Z,X)+2\lambda g(\phi_2Z,X)\}+\eta(Z)\{3\lambda g(\phi_2 X,Y)  \\
&\quad-3g(\phi_2^2 X,Y)-3g(\phi_3 X,\phi_3Y)-3\lambda g(\phi_2X,Y)-2g(\phi_3^2Y,X)-2\lambda g(\phi_2 Y,X)\}  \\
&=2\lambda\eta(X)\omega_2(Y,Z)+2\lambda\eta(Y)\omega_2(Z,X)+2\lambda\eta(Z)\omega_2(X,Y)-2\eta(Y)g(X,Z)\\
&\quad+2\eta(Z)g(X,Y)
\end{align*}
\begin{align*}
&=2\lambda (\eta\wedge\omega_2)(X,Y,Z)-2\eta(Y)g(X,Z)+2\eta(Z)g(X,Y)  \\
&=\frac{2}{3}d\omega_3(X,Y,Z)-2\eta(Y)g(X,Z)+2\eta(Z)g(X,Y) \\
&=-\frac{2}{3}d\Phi(X,Y,Z)-2\eta(Y)g(X,Z)+2\eta(Z)g(X,Y)
\end{align*}
thus proving \eqref{dPhi}, so that $(\phi_3,\xi,\eta,g)$ is a nearly Sasakian structure. Now, considering the structure tensor field $h=\nabla\xi+\phi_3$, we prove that $h=\lambda\phi_1$.
Indeed, by \eqref{nablaxiphi}, $\nabla_\xi\phi_3=\phi_3h$. On the other hand, using \eqref{dPhi} and \eqref{mainSU}, we have
\[g((\nabla_\xi\phi_3)Y,Z)=\frac{1}{3}d\omega_3(\xi,Y,Z)=\lambda(\eta\wedge\omega_2)(\xi,Y,Z)=\lambda g(\phi_2Y,Z).\]
Therefore, $\nabla_\xi\phi_3=\lambda \phi_2=\phi_3h$, which implies that $h=-\lambda\phi_3\phi_2=\lambda\phi_1$.

Finally, from \eqref{mainSU} one gets $d(\eta\wedge\omega_3)=-2\omega_1\wedge\omega_1$, so that the $SU(2)$-structure $(\eta,\omega_1,\omega_2,\omega_3)$ is nearly hypo.
\end{proof}

\begin{remark}
\emph{In \cite{BV} the authors determine explicit formulas for the scalar curvature and the Ricci tensor of the metric induced by an $SU(2)$-structure $(\eta,\omega_1,\omega_2,\omega_3)$ on a $5$-manifold in terms of the intrinsic torsion. For an $SU(2)$-structure satisfying \eqref{mainSU}, the only non-vanishing torsion forms are $\phi_1=2\lambda$, $\phi_3=-2$, $f_{12}=3$ and $f_{23}=-3\lambda$. Therefore, from (3.2) and Theorem 3.8 in \cite{BV}, it follows that $\mathrm{Ric}=4(1+\lambda^2)g$. We thus reacquire the result of Olszak (\cite{Olszak1}) stating that each $5$-dimensional nearly Sasakian  manifold is Einstein and of scalar curvature $s>20$. In particular,  \begin{equation}\label{scalar_S}
s=20(1+\lambda^2)
\end{equation}
implying that the constant $\lambda$ in \eqref{mainSU} is determined by the Riemannian geometry of the manifold.}
\end{remark}

Thus to any $5$-dimensional nearly Sasakian  manifold $(M,\phi,\xi,\eta,g)$ there are attached two other almost contact metric structures $(\phi_{1},\xi,\eta,g)$ and $(\phi_{2},\xi,\eta,g)$, with the same metric and characteristic vector field of $(\phi,\xi,\eta,g)$, such that the quaternionic relations \eqref{quaternionic} hold. In the following  we investigate the class to which these two supplementary almost contact metric structures belong.

To begin with, we recall a slight generalization of nearly Sasakian manifolds. Namely, a \emph{nearly $\alpha$-Sasakian manifold} is an almost contact metric manifold $(M,\phi,\xi,\eta,g)$ satisfying the following relation
\begin{equation*}
(\nabla_X\phi)Y+(\nabla_Y\phi)X=\alpha\left(2g(X,Y)\xi-\eta(X)Y-\eta(Y)X\right)
\end{equation*}
for some real number $\alpha\neq 0$.

\begin{lemma}\label{lemmanablaphi}
Let $(M,\phi,\xi,\eta,g)$ be a $5$-dimensional nearly Sasakian manifold. Then for all vector fields $X$, $Y$ on $M$ one has
\begin{gather}
(\nabla_{X}\phi)Y = \eta(X) \phi h Y - \eta(Y) (X + \phi h X) + g(X+\phi h X, Y)\xi,\label{nablaphi-completo}\\
(\nabla_{X}h)Y = \eta(X)\phi h Y - \eta(Y) (h^2 X + \phi h X) + g(h^{2}X + \phi h X,Y)\xi, \label{nablah-completo}\\
(\nabla_{X}\phi h)Y = g(\phi h^2 X - hX,Y)\xi + \eta(X)(\phi h^2 Y - hY) - \eta(Y) (\phi h^2 X - hX). \label{nablaphih-completo}
\end{gather}
\end{lemma}
\begin{proof}
The first equation follows by a direct computation using \eqref{dPhi}, \eqref{mainSU} and \eqref{quaternionic1}. Combining \eqref{Rxi} and \eqref{nablaphi-completo} one easily obtains \eqref{nablah-completo}. Finally, equations \eqref{nablaphi-completo} and \eqref{nablah-completo} imply \eqref{nablaphih-completo}.
\end{proof}

Now, from \eqref{nablah-completo} and \eqref{nablaphih-completo} it follows that
\begin{gather}
(\nabla_{X}h)Y + (\nabla_{Y}h)X = -\lambda^{2}\left(2g(X,Y)\xi - \eta(X)Y - \eta(Y)X\right) \label{nearly h}\\
(\nabla_{X}\phi h)Y + (\nabla_{Y}\phi h)X = 0. \label{nearly phih}
\end{gather}

Thus we can state the following result.

\begin{theorem}\label{main3}
Let  $(M,\phi,\xi,\eta,g)$  be  a  $5$-dimensional  nearly  Sasakian  manifold  and   let $(\phi_{i},\xi,\eta,g)$, $i\in\left\{1,2,3\right\}$, be the almost contact metric structures defined by the associated $SU(2)$-structure. Then $(\phi_{2},\xi,\eta,g)$ is nearly cosymplectic and $(\phi_{1},\xi,\eta,g)$ is nearly $\alpha$-Sasakian with $\alpha=-\lambda$.
\end{theorem}

We now find some applications of Theorem \ref{main2}, pointing out the relationship between nearly Sasakian geometry and Sasaki-Einstein manifolds.

\begin{corollary}\label{main5}
Each nearly Sasakian $5$-dimensional manifold carries a Sasaki-Einstein structure. \ Conversely, \ each Sasaki-Einstein $5$-manifold \ carries \ a \ $1$-parameter \ family \ of \ nearly \ Sasakian structures.
\end{corollary}
\begin{proof}
Let $M$ be a $5$-dimensional manifold.
Let $(\eta,\omega_1,\omega_2,\omega_3)$ be a nearly Sasakian $SU(2)$-structure on $M$, i.e. $(\eta,\omega_1,\omega_2,\omega_3)$ is an $SU(2)$-structure satisfying \eqref{mainSU} for some real number $\lambda\ne0$.
Put
\begin{align}\label{Sasaki-E}
\tilde\eta &:=\sqrt{1+\lambda^2}\,\eta,\\
\tilde\omega_1&:=\sqrt{1+\lambda^2}\,(\omega_1+\lambda\omega_3),\nonumber\\
\tilde\omega_2&:=(1+\lambda^2)\,\omega_2,\nonumber\\
\tilde\omega_3&:=\sqrt{1+\lambda^2}\,(\omega_3-\lambda\omega_1).\nonumber
\end{align}
One can easily check that $\tilde\omega_i\wedge\tilde\omega_j=\delta_{ij}\tilde v$, where $\tilde v=(1+\lambda^2)^2\omega_i\wedge\omega_i$, and $\tilde\eta\wedge\tilde v\ne0$. Furthermore, suppose that $X\lrcorner\,\tilde\omega_1=Y\lrcorner\,\tilde\omega_2$ for some vector fields $X,Y$. Let $\{(\phi_i,\xi,\eta,g)\}_{i\in\{1,2,3\}}$ be the almost contact metric structures associated to $(\eta,\omega_i)$. Then
\[\phi_1X+\lambda \phi_3X=\sqrt{1+\lambda^2}\,\phi_2Y\]
and applying $\phi_2$, we have $-\phi_3X+\lambda \phi_1X=\sqrt{1+\lambda^2}\,(-Y+\eta(Y)\xi)$.
Then,
\begin{align*}
\tilde\omega_3(X,Y)&=\sqrt{1+\lambda^2}\,g(\phi_3X-\lambda \phi_1X,Y)\\&=(1+\lambda^2)(g(Y,Y)-\eta(Y)^2)\\
&=(1+\lambda^2)g(\phi_iY,\phi_iY)\geq0.
\end{align*}
It is straightforward to verify that the $SU(2)$-structure $(\tilde\eta,\tilde\omega_1,\tilde\omega_2,\tilde\omega_3)$
 satisfies \eqref{S-E} and thus it is a Sasaki-Einstein structure.

Analogously, given a Sasaki-Einstein structure $(\tilde\eta,\tilde\omega_1,\tilde\omega_2,\tilde\omega_3)$ on $M$, for any real number $\lambda\ne 0$, one can define the nearly Sasakian structure
\begin{align}\label{Sasaki-Enearly}
\eta &:=\frac{1}{\sqrt{1+\lambda^2}}\,\tilde\eta,\\
\omega_1&:=\frac{1}{\sqrt{1+\lambda^2}(1+\lambda^2)}\,(\tilde\omega_1-\lambda\tilde\omega_3),\nonumber\\
\omega_2&:=\frac{1}{1+\lambda^2}\,\tilde\omega_2,\nonumber\\
\omega_3&:=\frac{1}{\sqrt{1+\lambda^2}(1+\lambda^2)}\,(\lambda\tilde\omega_1+\tilde\omega_3).\nonumber
\end{align}
\end{proof}

Corollary \ref{main5} provides a way of finding new examples of nearly-Sasakian manifolds. In particular, each Sasaki-Einstein metric of the infinite family of Sasakian structures on $S^{2}\times S^{3}$   recently discovered in \cite{MarSpa} gives examples of nearly Sasakian structures.
\medskip

We point out that, in terms of almost contact metric structures, the Sasaki-Einstein structure $(\tilde\phi,\tilde\xi,\tilde\eta,\tilde g)$ associated to any $5$-dimensional nearly Sasakian manifold $(M,\phi,\xi,\eta,g)$  is given by
\begin{equation}\label{deformation}
\tilde\phi=\frac{1}{\sqrt{1+\lambda^2}}\,(\phi-h),\quad \tilde\xi=\frac{1}{\sqrt{1+\lambda^2}}\,\xi,\quad \tilde\eta=\sqrt{1+\lambda^2}\,\eta,\quad\tilde g=(1+\lambda^2)g.
\end{equation}
The scalar curvatures $s$ and $\tilde s$ of $g$ and $\tilde g$, respectively, are related by $s=(1+\lambda^2)\tilde s$, coherently with \eqref{scalar_S}, since the scalar curvature of a $5$-dimensional Sasaki-Einstein structure is $\tilde s =20$.

\begin{remark}
\emph{One can find a more direct proof that the structure $(\tilde\phi,\tilde\xi,\tilde\eta,\tilde g)$ in \eqref{deformation} is Sasakian. Indeed,}
\[
d\eta(X,Y)= g(Y,\nabla_X\xi)-g(X,\nabla_Y\xi)=2 g(X,(\phi-h)Y)
\]
\emph{and thus} $d\tilde\eta(X,Y)=2\tilde g(X,\tilde\phi Y)$,
\emph{implying that $(\tilde\phi,\tilde\xi,\tilde\eta,\tilde g)$ is a contact metric structure. Applying \eqref{RXYxi}, a straightforward computation yields
\[\tilde R(X,Y)\tilde\xi=R(X,Y)\tilde\xi=\tilde\eta(Y)X-\tilde\eta(X)Y,\]
which ensures that the structure is Sasakian (\cite[Proposition 7.6]{BLAIR}).}
\end{remark}

\begin{remark}\label{main4}
\emph{Explicitly, the almost contact metric structures $(\tilde\phi_{i},\tilde\xi,\tilde\eta,\tilde{g})$ associated to the Sasaki-Einstein $SU(2)$-structure   \eqref{Sasaki-E} is given by}
\begin{align*}
\tilde\phi_{1}&:=\frac{1}{\sqrt{1+\lambda^{2}}}\left(\frac{1}{\lambda}h + \lambda \phi\right) = { \frac{1}{3\lambda\sqrt{1+\lambda^{2}}} } {\mathcal L}_{\xi}\phi h\,, \\
\tilde\phi_{2}&:=\frac{1}{\lambda}\phi h = \frac{1}{3\lambda} {\mathcal L}_{\xi}\phi\,,\\
\tilde\phi_{3}&:=\frac{1}{\sqrt{1+\lambda^{2}}}\left(\phi - h \right).
\end{align*}
\emph{Using Lemma \ref{lemmanablaphi}  one can prove that $(\tilde\phi_{1},\tilde\xi,\tilde\eta,\tilde{g})$ and $(\tilde\phi_{2},\tilde\xi,\tilde\eta,\tilde{g})$ are nearly cosymplectic. Actually we will see in Corollary \ref{SE-nearly} that this result holds for \emph{any} Sasaki-Einstein $SU(2)$-structure.}
\end{remark}

\begin{remark}
\emph{The  Sasaki-Einstein structure \eqref{Sasaki-E} defined on the nearly Sasakian manifold $M$ determines an integrable $SU(3)$-structure on $M\times\mathbb{R}_+$ which is given by the closed forms (see \cite{CS})}
\begin{align*}
F&= \sqrt{1+\lambda^2}\left\{t^2\omega_3+t\eta\wedge dt-\lambda t^2\omega_1\right\},\\
 \Psi_+&=(1+\lambda^2)\left\{t^2(t\omega_1\wedge\eta-\omega_2\wedge dt)+\lambda t^3\omega_3\wedge\eta\right\},\\
 \Psi_-&=\sqrt{1+\lambda^2}\left\{t^2(t\omega_2\wedge \eta+\omega_1\wedge dt)+\lambda t^2\omega_3\wedge dt+\lambda^2 t^3\omega_2\wedge\eta\right\}.
\end{align*}
\emph{In particular, the K\"ahler and Ricci-flat structure $(G,J)$ of the metric cone is given by}
\begin{align*}
G&=dt^2+(1+\lambda^2)t^2g,\\
JX&=\frac{1}{\sqrt{1+\lambda^2}}\,(\phi X-hX)+\sqrt{1+\lambda^2}\,\eta(X)\Upsilon,\\
J\Upsilon&=-\frac{1}{\sqrt{1+\lambda^2}}\,\xi,\quad \Upsilon=t\frac{\partial}{\partial t}\,.
\end{align*}

\emph{On the other hand, following \cite[Theorem 3.7 and Corollary 3.8]{FISU},  one can define on the product $M\times[0,\pi]$ an $SU(3)$-structure which is nearly K\"ahler for $0<t<\pi$:}
\begin{align*}
F&= \sqrt{1+\lambda^2}\left\{\sin^2t(\sin t\,\omega_1+\cos t\,\omega_3)+\sin t\,\eta\wedge dt+\lambda\sin^2 t(\sin t\,\omega_3-\cos t\,\omega_1)\right\},\\
 \Psi_+&=\sqrt{1+\lambda^2}\left\{\sin^3t\,\eta\wedge\omega_2+\sin^2t(\cos t\,\omega_1-\sin t\,\omega_3)\wedge dt\right.\\
 &\quad\quad\quad\quad\quad\quad\left.+\lambda^2\sin^3 t\,\eta\wedge\omega_2+\lambda\sin^2 t\,(\cos t\,\omega_3+\sin t\,\omega_1)\wedge dt\right\},\\
 \Psi_-&=(1+\lambda^2)\left\{\sin^3t\,(-\cos t\,\omega_1+\sin t\,\omega_3)\wedge\eta+\sin^2t\,\omega_2\wedge dt\right.\\
 &\quad\quad\quad\quad\quad\quad\left.-\lambda\sin^3 t(\cos t\,\omega_3+\sin t\,\omega_1)\wedge\eta+\lambda^2\sin^2 t\,\omega_2\wedge dt\right\}.
\end{align*}
\emph{In this case, the Riemannian metric and the almost complex structure are given by}
\begin{align*}
G&=dt^2+(1+\lambda^2)\sin^2t\,g,\\
JX&=\frac{1}{\sqrt{1+\lambda^2}}\,\left\{\sin t\left(\frac{1}{\lambda}h X+\lambda\phi X\right)+\cos t(\phi X-hX)\right\}+\sqrt{1+\lambda^2}\,\eta(X)\Upsilon,\\
J\Upsilon&=-\frac{1}{\sqrt{1+\lambda^2}}\,\xi,\quad \Upsilon=\sin t\frac{\partial}{\partial t}\,.
\end{align*}
\end{remark}

Corollary \ref{main5} together with Theorem \ref{main1} have an interesting application for a general nearly Sasakian manifold in any dimension.

\begin{corollary}
Every nearly Sasakian manifold is a contact manifold.
\end{corollary}
\begin{proof}
Let $M$ be a nearly Sasakian manifold of dimension $2n+1$ with structure $(\phi,\xi,\eta,g)$. With the notation used in Section \ref{foliationsection} preliminarly we prove that for any $X\in{\mathcal D}(-\lambda_{i}^{2})$, $Y\in{\mathcal D}(-\lambda_{j}^{2})$
\begin{equation}\label{contatto1}
d\eta(X,Y)=0,
\end{equation}
for each $i,j\in\left\{1,\ldots,r\right\}$, $i\neq j$. Indeed,
\[d\eta(X,Y)=g(Y,\nabla_{X}\xi)-g(X,\nabla_{Y}\xi)=2g(X,\phi Y) + 2g(hX,Y)=0\]
since the operators $\phi$ and $h$ preserve ${\mathcal D}(-\lambda_{i}^2)$ and the distributions ${\mathcal D}(-\lambda_{i}^2)$ and ${\mathcal D}(-\lambda_{j}^2)$ are mutually orthogonal. In a similar way one can prove that for any  $X\in{\mathcal D}(-\lambda_{i}^{2})$ and $Z\in{\mathcal D}(0)$
\begin{equation}\label{contatto1b}
d\eta(X,Z)=0.
\end{equation}
Now, fix a point $x \in M$. By a) in Theorem \ref{main1} there exists a basis $\{\xi_{x}, e_{1}, \ldots, e_{2p}\}$  of ${\mathcal D}_{x}(0)$ such that
\begin{equation}\label{contatto2}
\eta\wedge (d\eta)^{p}(\xi_{x}, e_{1}, \ldots, e_{2p})\neq 0.
\end{equation}
By  b)  in Theorem \ref{main1}  and Corollary \ref{main5}, \ for each $i\in\left\{1,\ldots,r\right\}$ \ one can find a basis \ $\{\xi_{x}, v_{1}^{i}, v_{2}^{i}, v_{3}^{i}, v_{4}^{i}\}$ of ${\mathcal D}_{x}(-\lambda_{i}^{2})$ such that
\begin{equation}\label{contatto3}
\eta\wedge (d\eta)^{2}(\xi_{x}, v_{1}^{i}, v_{2}^{i}, v_{3}^{i}, v_{4}^{i})\neq 0.
\end{equation}
Then by \eqref{contatto1}, \eqref{contatto1b}, \eqref{contatto2} and \eqref{contatto3} one has
\begin{align*}
\eta \wedge (d\eta)^{n} & \left(  \xi_{x}, e_{1}, \ldots, e_{2p},  v_{1}^{1}, v_{2}^{1}, v_{3}^{1}, v_{4}^{1}, \ldots, v_{1}^{r}, v_{2}^{r}, v_{3}^{r}, v_{4}^{r}\right) \\
&= \eta(\xi_{x}) (d\eta)^{p}(e_{1}, \ldots, e_{2p}) (d\eta)^{2}(v_{1}^{1}, v_{2}^{1}, v_{3}^{1}, v_{4}^{1}) \cdots (d\eta)^{2}(v_{1}^{r}, v_{2}^{r}, v_{3}^{r}, v_{4}^{r}) \neq 0.
\end{align*}
\end{proof}

Theorem \ref{main3} shows that any $5$-dimensional nearly Sasakian manifold is naturally endowed with a nearly cosymplectic structure, via the nearly Sasakian $SU(2)$-structure \eqref{mainSU}. On the other hand, as pointed out in Remark \ref{main4}, the deformed $SU(2)$-structure \eqref{Sasaki-E}, which is Sasaki-Einstein, carries two other nearly cosymplectic structures. Thus we devote the next section to further investigate nearly cosymplectic structures on $5$-dimensional manifolds: we show that they are nothing but deformations of Sasaki-Einstein $SU(2)$-structures.

\section{Sasaki-Einstein $SU(2)$-structures and nearly cosymplectic manifolds}

First, we remark that in any $5$-dimensional nearly cosymplectic manifold $(M,\phi, \xi,\eta,g)$ the vanishing of the operator $h$ defined in \eqref{nablaxi_c} provides a necessary and sufficient condition for the structure to be coK\"{a}hler. Indeed, if $h=0$ then the distribution $\mathcal D$ orthogonal to $\xi$ is integrable with totally geodesic leaves; the manifold $M$ turns out to be locally isometric to the Riemannian product $N\times \mathbb{R}$, where $N$ is an integral submanifold of ${\mathcal D}=\ker(\eta)$ endowed with a nearly K\"ahler structure $(g,J)$ induced by the structure tensors $(g,\phi)$. On the other hand, it is known that $4$-dimensional nearly K\"ahler manifolds are K\"ahler (see \cite[Theorem 5.1]{Gray}), and this implies that $(\phi, \xi,\eta,g)$ is a coK\"{a}hler structure.
\medskip

Let $(M,\phi, \xi,\eta,g)$ be a $5$-dimensional nearly cosymplectic manifold.
Let $X$ be a local eigenvector field of the operator $h^2$ with eigenvalue $\mu\ne0$. Then $\{\xi, X, \phi X, hX, h\phi X\}$ is a local orthogonal frame, and $\phi X, hX, h\phi X$ are eigenvector fields of $h^2$ with the same eigenvalue $\mu$. Then one has $h^2=\mu(I-\eta\otimes\xi)$ which, together with \eqref{tr}, implies that $\mu$ is constant. On the other hand, being $h$ skew-symmetric, necessarily $\mu<0$. We put $\mu=-\lambda^2$, $\lambda\ne 0$. In fact $M$ is endowed with an $SU(2)$-structure, as described in the following theorem.

\begin{theorem}
A nearly cosymplectic  structure on a $5$-dimensional manifold is equivalent to an $SU(2)$-structure $(\eta,\omega_1,\omega_2,\omega_3)$ satisfying
\begin{equation}\label{mainSUc}
d\eta=-2\lambda\omega_3,\qquad d\omega_1=3\lambda\eta\wedge\omega_2,\qquad d\omega_2=-3\lambda\eta\wedge\omega_1
\end{equation}
for some real number $\lambda\ne0$. These $SU(2)$-structures are hypo.
\end{theorem}
\begin{proof}
Let $(M,\phi,\xi,\eta,g)$ be a nearly cosymplectic $5$-manifold. The operator $h$ satisfies
\begin{equation*}\label{nearlycos-h}
h^2=-\lambda^2(I-\eta\otimes\xi),
\end{equation*}
for some real number $\lambda\ne0$. Arguing as in Theorem \ref{main2}, the tensor fields
\[\phi_1:=-\frac{1}{\lambda}\phi h,\qquad \phi_2=\phi,\qquad \phi_3:=-\frac{1}{\lambda} h\]
determine  an $SU(2)$-structure $(\eta,\omega_1,\omega_2,\omega_3)$, with $\omega_i(X,Y):=g(\phi_i X,Y)$.
We prove that this structure satisfies \eqref{mainSUc}. Indeed, using \eqref{nablaxi_c}, a simple computation shows that
\[
d\eta(X,Y)=2 g(hX,Y)= -2\lambda\omega_3(X,Y).\]
By \eqref{dPhi_c}, we have
\[
d\omega_2(X,Y,Z)= 3g((\nabla_X\phi)Y,Z).
\]
For $X=\xi$, using \eqref{nearlycos-nablaxiphi}, we get
\[d\omega_2(\xi,Y,Z)=3g((\nabla_\xi\phi)Y,Z)=3g(\phi hY,Z)=-3\lambda\omega_1(Y,Z).\]
Equation \eqref{nablaphi_hc} implies that for every vector fields $X,Y,Z$ orthogonal to $\xi$, $g((\nabla_X\phi)Y,Z)=0$ and thus $d\omega_2(X,Y,Z)=0$. Therefore
$d\omega_2=-3\lambda\eta\wedge\omega_1$. In particular we get $d(\eta\wedge\omega_1)=0$ and hence, by \eqref{SU2},
\[\eta\wedge d\omega_1=d\eta\wedge \omega_1=0.\]
Therefore, for every vector fields $X,Y,Z$ orthogonal to $\xi$,
\[d\omega_1(X,Y,Z)=(\eta\wedge d\omega_1)(\xi,X,Y,Z)=0.\]
Now, from \eqref{nablah_c} we have $\nabla_\xi h=0$, and thus, by \eqref{nearlycos-nablaxiphi},
\[\nabla_\xi(\phi h)=(\nabla_\xi\phi) h=\phi h^2=-\lambda^2\phi.\]
Hence, for every vector fields $Y$, $Z$, using also \eqref{nablaxi_c}, we compute
\[\lambda d\omega_1(\xi,Y,Z)=-g((\nabla_\xi\phi h)Y,Z)-g((\nabla_Y\phi h)Z,\xi)-g((\nabla_Z\phi h)\xi,Y)=3\lambda^2g(\phi Y,Z)\]
which implies $d\omega_1(\xi,Y,Z)=3\lambda\omega_2(Y,Z)$.
Consequently, $d\omega_1=3\lambda\eta\wedge\omega_2$ and this completes the proof of \eqref{mainSUc}.

As for the converse, assume that $M$ is a $5$-manifold with an $SU(2)$-structure satisfying \eqref{mainSUc} for some real number $\lambda\ne0$. Consider the associated almost contact metric structures $(\phi_i,\xi,\eta,g)$, $i\in\{1,2,3\}$. By using \eqref{nablaphi} and \eqref{mainSUc}, a straightforward computation shows that the covariant derivative of $\phi_2$ is given by:
\[g((\nabla_X\phi_2)Y,Z)=-\frac{1}{3}\,d\Phi(X,Y,Z)\]
so that $(\phi_2,\xi,\eta,g)$ is a nearly cosymplectic structure. The associated operator $h=\nabla\xi$ coincides with $-\lambda\phi_3$. Indeed, applying \eqref{mainSUc},
 \[g((\nabla_\xi\phi_2)Y,Z)=\frac{1}{3}\,d\omega_2(X,Y,Z)=-\lambda(\eta\wedge\omega_1)(\xi,Y,Z)=-\lambda g(\phi_1Y,Z),\]
and thus $\nabla_\xi\phi_2=-\lambda\phi_1$. On the other hand, by \eqref{nearlycos-nablaxiphi}, $\nabla_\xi\phi_2=\phi_2h$. Hence, $h=\lambda\phi_2\phi_1=-\lambda\phi_3$.

Finally, form \eqref{mainSUc}  the forms $\omega_3$, $\eta\wedge\omega_1$, $\eta\wedge\omega_2$ are closed so that the structure $(\eta,\omega_1,\omega_2,\omega_3)$ is hypo.
\end{proof}

Note that if $(\eta,\omega_1,\omega_2,\omega_3)$ is an $SU(2)$-structure satisfying \eqref{mainSUc} and $(\phi_i,\xi,\eta,g)$, $i\in\{1,2,3\}$, are the associated almost contact metric structures, then applying \eqref{nablaphi} one can  verify that also $(\phi_1,\xi,\eta,g)$ is a nearly cosymplectic structure, while the covariant derivative of $\phi_3$ is given by
\[(\nabla_X\phi_3)Y=\lambda(g(X,Y)\xi-\eta(Y)X),\]
and thus $(\phi_3,\xi,\eta,g)$ is a $\lambda$-Sasakian structure. In particular, for $\lambda=1$, equations \eqref{mainSUc} reduce to the equations of a Sasaki-Einstein structure, so that we deduce the following results.

\begin{corollary}\label{SE-nearly}
Let $(\eta,\omega_1,\omega_2,\omega_3)$ be an $SU(2)$-structure satisfying the Sasaki-Einstein equations \eqref{S-E}. Let $(\phi_i,\xi,\eta,g)$, $i\in\{1,2,3\}$, be the associated almost contact metric structures. Then, for $i=1,2$, $(\phi_i,\xi,\eta,g)$ is a  nearly cosymplectic structure.
\end{corollary}

\begin{corollary}\label{nearlycos-se}
Each nearly cosymplectic $5$-dimensional manifold carries a Sasaki-Einstein structure. Conversely, each Sasaki-Einstein $5$-manifold carries a $1$-parameter family of nearly cosymplectic  structures.
\end{corollary}
\begin{proof}
Let $M$ be a $5$-dimensional manifold and let $(\eta,\omega_1,\omega_2,\omega_3)$ be an $SU(2)$-structure satisfying \eqref{mainSUc} for some real number $\lambda\ne0$.
Put
\begin{equation}\label{Sasaki-Ec}
\tilde\eta :=\lambda\eta,\quad\tilde\omega_1:=\lambda^2\omega_1,\quad\tilde\omega_2:=\lambda^2\omega_2,\quad\tilde\omega_3:=\lambda^2\omega_3.
\end{equation}
Obviously $(\tilde\eta,\tilde\omega_1,\tilde\omega_2,\tilde\omega_3)$ is an $SU(2)$-structure and one can easily check that it satisfies \eqref{S-E}.
Conversely, given a Sasaki-Einstein structure $(\tilde\eta,\tilde\omega_1,\tilde\omega_2,\tilde\omega_3)$ on $M$, for any real number $\lambda\ne 0$, one can define the $SU(2)$-structure
\begin{equation*}\label{Sasaki-Enearlyc}
\eta :=\frac{1}{\lambda}\,\tilde\eta,\quad\omega_1:=\frac{1}{\lambda^2}\,\tilde\omega_1,\quad
\omega_2:=\frac{1}{\lambda^2}\,\tilde\omega_2,\quad
\omega_3:=\frac{1}{\lambda^2}\,\tilde\omega_3,
\end{equation*}
which satisfies \eqref{mainSUc}.
\end{proof}

In terms of almost contact metric structures, the Sasaki-Einstein structure $(\tilde\phi,\tilde\xi,\tilde\eta,\tilde g)$ attached to any  $5$-dimensional nearly cosymplectic  manifold $(M,\phi,\xi,\eta,g)$, stated by Corollary \ref{nearlycos-se}, is given by
\begin{equation*}
\tilde\phi=-\frac{1}{\lambda}h,\quad \tilde\xi=\frac{1}{\lambda}\,\xi,\quad \tilde\eta=\lambda\eta,\quad\tilde g=\lambda^2g.
\end{equation*}
In particular, the scalar curvatures $s$ and $\tilde s$ of $g$ and $\tilde g$, respectively, are related by
\begin{equation}\label{scalar_c}
s=\lambda^2\tilde s=20\lambda^2.
\end{equation}
Therefore we have the following

\begin{theorem}
Every nearly cosymplectic (non-coK\"{a}hler) $5$-dimensional manifold is Einstein with positive scalar curvature.
\end{theorem}

\section{Hypersurfaces of nearly K\"ahler manifolds}\label{hypersurfaces}

Let $(N,J,\tilde g)$ be an almost Hermitian manifold of dimension $2n+2$. Let $\iota:M\to N$ be a $\mathcal{C}^\infty$ orientable hypersurface and $\nu$ a unit normal vector field. As it is known (see \cite[Section 4.5.2]{BLAIR}) on $M$ it is induced a natural almost contact metric structure $(\phi, \xi,\eta,g)$ given by
\[J\iota_*X=\iota_*\phi X+\eta(X)\nu,\qquad J\nu=-\iota_*\xi,\qquad g=\iota^*\tilde g.\]

We recall now the following fundamental results providing necessary and sufficient conditions for a hypersurface of a nearly K\"ahler manifold to be nearly cosymplectic or nearly Sasakian.

\begin{theorem}[\cite{BLAIR_cos}]\label{hyp_cosym}
Let $M$ be a hypersurface of a nearly K\"ahler manifold $(N,J,g')$. Then the induced almost contact metric manifold $(\phi,\xi,\eta,g)$ is nearly cosymplectic if and only if the second fundamental form is given by $\sigma=\beta(\eta\otimes\eta)\nu$ for some function $\beta$.
\end{theorem}

\begin{theorem}[\cite{BlairSY}]\label{hyp_sas}
Let $M$ be a hypersurface of a nearly K\"ahler manifold $(N,J, g')$. Then the induced almost contact metric manifold $(\phi,\xi,\eta,g)$ is nearly Sasakian if and only if the second fundamental form is given by $\sigma=(-g+\beta(\eta\otimes\eta))\nu$ for some function $\beta$.
\end{theorem}

 Concerning $6$-dimensional nearly K\"ahler manifolds, we shall further investigate the $SU(2)$-structure induced on hypersurfaces satisfying the conditions stated in Theorems \ref{hyp_cosym} and \ref{hyp_sas}. First recall that, as proved in \cite{Gray}, any $6$-dimensional nearly K\"ahler non-K\"ahler manifold $(N,J, g')$ is Einstein and of constant type, i.e. it satisfies
\begin{equation}\label{nearlyK}
\|(\nabla'_XJ)Y\|^2=\frac{s'}{30}\left(\|X\|^2\cdot\|Y\|^2-g'(X,Y)^2-g'(X,JY)^2\right)
\end{equation}
where $\nabla'$ is the Levi-Civita connection and $s'>0$ is the scalar curvature of $g'$.

\begin{theorem}\label{hypersurface_c}
Let $(N,J,g')$ be a $6$-dimensional nearly K\"ahler non-K\"ahler manifold and let $M$ be a hypersurface such that the second fundamental form is given by $\sigma=\beta(\eta\otimes\eta)\nu$ for some function $\beta$. Let $(\phi,\xi,\eta,g)$ be the induced nearly cosymplectic structure on $M$ and $(\eta,\omega_1,\omega_2,\omega_3)$ the associated $SU(2)$-structure satisfying \eqref{mainSUc}. Then the operator $h$ coincides with the covariant derivative $\nabla'_\nu J$ and the constant $\lambda$ satisfies
\begin{equation*}
\lambda^2=\frac{s'}{30}
\end{equation*}
Therefore, the scalar curvature of the Einstein Riemannian metric $g$ is $s=\frac{2}{3}s'$.
\end{theorem}
\begin{proof}
First notice that the hypothesis on the second fundamental form implies that, for any vector fields $X,Y\in\frak{X}(M)$,
\[\nabla'_XY=\nabla_XY+\beta\eta(X)\eta(Y)\nu,\qquad\nabla'_X\nu=-\beta\eta(X)\xi.\]
Therefore,
\begin{align*}
(\nabla'_\nu J)X &= -(\nabla'_X J)\nu\\
&= \nabla'_X\xi+J(\nabla'_X\nu)\\
&= \nabla_X\xi+\beta\eta(X)\nu-\beta\eta(X)J\xi\\
&= hX.
\end{align*}
Now, taking a unit vector field $X$ orthogonal to $\xi$ and applying \eqref{nearlyK}, we have
\[\|hX\|^2=\|(\nabla'_\nu J)X\|^2=\frac{s'}{30}.\]
On the other hand, being $h^2=-\lambda^2(I-\eta\otimes\xi)$, then $\|hX\|^2=-g(h^2X,X)=\lambda^2$.
The assertion on the scalar curvature is consequence of \eqref{scalar_c}.
\end{proof}

Under the hypothesis of the above theorem, applying the deformation \eqref{Sasaki-Ec} to the $SU(2)$-structure $(\eta,\omega_1,\omega_2,\omega_3)$, one obtains a Sasaki-Einstein structure. Therefore,
\begin{corollary}
Every hypersurface of a $6$-dimensional nearly K\"ahler non-K\"ahler manifold such that the second fundamental form is proportional to $(\eta\otimes\eta)\nu$ carries a Sasaki-Einstein structure.
\end{corollary}
The above Corollary generalizes Lemma 2.1 of \cite{FISU} concerning totally geodesic hypersurfaces of nearly K\"ahler manifolds.

Analogously, we prove the following

\begin{theorem}
Let $(N,J,g')$ be a $6$-dimensional nearly K\"ahler non-K\"ahler manifold and let $M$ be a hypersurface such that the second fundamental form is given by $\sigma=(-g+\beta(\eta\otimes\eta))\nu$ for some function $\beta$. Let $(\phi,\xi,\eta,g)$ be the induced nearly Sasakian structure on $M$ and $(\eta,\omega_1,\omega_2,\omega_3)$ the associated $SU(2)$-structure  satisfying \eqref{mainSU}. Then the operator $h$ coincides with the covariant derivative $\nabla'_\nu J$ and the constant $\lambda$ satisfies
\begin{equation*}
\lambda^2=\frac{s'}{30}
\end{equation*}
Therefore, the scalar curvature of the Einstein Riemannian metric $g$ is $s=20+\frac{2}{3}s'$.
\end{theorem}
\begin{proof}
For every vector fields $X,Y\in\frak{X}(M)$, we have
\[\nabla'_XY=\nabla_XY-g(X,Y)\nu+\beta\eta(X)\eta(Y)\nu,\qquad\nabla'_X\nu=X-\beta\eta(X)\xi.\]
Therefore,
\begin{align*}
(\nabla'_\nu J)X &= -(\nabla'_X J)\nu\\
&= \nabla'_X\xi+J(\nabla'_X\nu)\\
&= \nabla_X\xi-\eta(X)\nu+\beta\eta(X)\nu+J X-\beta\eta(X)J\xi\\
&= -\phi X+hX+\phi X\\
&= hX.
\end{align*}
Taking a unit vector field $X$ orthogonal to $\xi$ and applying \eqref{nearlyK}, we have $\|hX\|^2=\frac{s'}{30}$.
On the other hand, $\|hX\|^2=-g(h^2X,X)=\lambda^2$.
The assertion on the scalar curvature is consequence of \eqref{scalar_S}.
\end{proof}
In this case, applying the deformation \eqref{Sasaki-E} to the $SU(2)$-structure $(\eta,\omega_1,\omega_2,\omega_3)$, we obtain a Sasaki-Einstein structure. Therefore,
\begin{corollary}
Every hypersurface of a $6$-dimensional nearly K\"ahler non-K\"ahler manifold  such that the second fundamental form is given by $\sigma=(-g+\beta(\eta\otimes\eta))\nu$, for some function $\beta$, carries a Sasaki-Einstein structure.
\end{corollary}
In particular the above Corollary holds for totally umbilical hypersurfaces of nearly K\"ahler manifolds with shape operator $A=-I$.

\begin{example}
\emph{We recall two basic examples of $5$-dimensional nearly cosymplectic and nearly Sasakian manifolds (\cite{BLAIR_cos, BlairSY}). First consider $\mathbb{R}^7$ as the imaginary part of the Cayley numbers $\mathbb{O}$,  with the product vector $\times$ induced by the Cayley product. Let $S^6$ be the unit sphere in $\mathbb{R}^7$ and $N=\sum_{i=1}^7x^i\frac{\partial}{\partial x^i}$  the unit outer normal. One can define an almost complex structure $J$ on $S^6$ by $JX=N\times X$. It is well known that this almost complex structure is nearly K\"ahler (non-K\"ahler) with respect to the induced Riemannian metric.}

\emph{Consider $S^5$ as a totally geodesic hypersurface of $S^6$ defined by $x^7=0$ with unit normal $\nu=-\frac{\partial}{\partial x^7}$. Let $(\phi,\xi,\eta,g)$ be the induced almost contact metric structure on $S^5$, with
\[\xi=-J\nu=N\times\frac{\partial}{\partial x^7}=x^1\frac{\partial}{\partial x^6}-x^2\frac{\partial}{\partial x^5}-x^3\frac{\partial}{\partial x^4}+x^4\frac{\partial}{\partial x^3}
 +x^5\frac{\partial}{\partial x^2}-x^6\frac{\partial}{\partial x^1},\]
and $\eta$ given by the restriction of $x^1dx^6-x^6dx^1+x^5dx^2-x^2dx^5+x^4dx^3-x^3dx^4$ to $S^5$. This almost contact metric structure is nearly cosymplectic non-coK\"ahler. Considering the associated $SU(2)$-structure $(\eta,\omega_1,\omega_2,\omega_3)$ satisfying \eqref{mainSUc}, we have $\lambda^2=1$ since the scalar curvature of $S^6$ is $s'=30$. Coherently with Theorem \ref{hypersurface_c}, the scalar curvature of $S^5$ is $s=20$.}

\emph{Now consider $S^5$ as a totally umbilical hypersurface of $S^6$ defined by $x^7=\frac{\sqrt{2}}{2}$, with unit normal at each point $x$ given by $\nu=x-\sqrt{2}\frac{\partial}{\partial x^7}=\sum_{i=1}^{6}x^i\frac{\partial}{\partial x^i}-\frac{\sqrt{2}}{2}\frac{\partial}{\partial x^7}$, so that the shape operator is $A=-I$. Let $(\phi,\xi,\eta,g)$ be the induced almost contact metric structure, where
\[\xi=-J\nu=\sqrt{2}\left(x^1\frac{\partial}{\partial x^6}-x^2\frac{\partial}{\partial x^5}-x^3\frac{\partial}{\partial x^4}+x^4\frac{\partial}{\partial x^3}
 +x^5\frac{\partial}{\partial x^2}-x^6\frac{\partial}{\partial x^1}\right)\,,\]
and $\eta$ given by the restriction of $\sqrt{2}\left(x^1dx^6-x^6dx^1+x^5dx^2-x^2dx^5+x^4dx^3-x^3dx^4\right)$ to $S^5$. This structure is nearly Sasakian, but not Sasakian and again, taking into account the associated $SU(2)$-structure satisfying \eqref{mainSU}, the constant $\lambda$ satisfies $\lambda^2=1$. The scalar curvature of the hypersurface is $40$, coherently with the fact that it has constant sectional curvature $2$.}
\end{example}

\section{Canonical connections on nearly Sasakian manifolds}
It is well known that nearly Kahler manifolds are endowed with a canonical Hermitian connection $\bar\nabla$, called \emph{Gray connection}, defined by
\begin{equation*}
\bar\nabla_{X}Y = \nabla_{X}Y + \frac{1}{2}(\nabla_{X}J)JY,
\end{equation*}
which is the unique Hermitian connection with totally skew-symmetric torsion. To the knowledge of the authors there does not exist any canonical connection, analogous to $\bar\nabla$, in the context of nearly Sasakian geometry. In particular, in \cite{FrIv} Friedrich and Ivanov proved that an almost contact metric manifold $(M,\phi,\xi,\eta,g)$ admits a (unique) linear connection with totally skew-symmetric torsion parallelizing all the structure tensors, if and only if $\xi$ is Killing and the tensor $N_\phi$ is totally skew-symmetric. Using this result, we prove the following
\begin{proposition}
A nearly Sasakian manifold $(M,\phi,\xi,\eta,g)$ admits a linear connection with totally skew-symmetric torsion parallelizing all the structure tensors if and only if it is Sasakian.
\end{proposition}
\begin{proof}
Recall that the tensor field $N_\phi$ is also given by
\[N_{\phi}(X,Y)=(\nabla_{\phi X}\phi)Y-(\nabla_{\phi Y}\phi)X+(\nabla_X\phi)\phi Y-(\nabla_Y\phi)\phi X+\eta(X)\nabla_Y\xi-\eta(Y)\nabla_X\xi.\]
Setting $N(X,Y,Z):=g(N_{\phi}(X,Y),Z)$, a straightforward computation using the above formula, \eqref{nablaxi} and \eqref{nablaxiphi}, gives
\[N(X,Y,\xi)+N(X,\xi,Y)=g(hX,Y).\]
Hence, if $N_\phi$ is totally-symmetric, then $h=0$ and the structure is Sasakian.
\end{proof}
Thus it makes sense to find adapted connections which can be useful in the study of nearly Sasakian manifolds. We have the following theorem.

\begin{theorem}
Let $(M,\phi,\xi,\eta,g)$ be a nearly Sasakian manifold. Fix a real number $r$. Then, there exists a unique linear connection $\bar\nabla$ which parallelizes all the structure tensors and such that the torsion tensor $\bar T$ of $\bar\nabla$ satisfies the following conditions:
\begin{itemize}
\item[1)] $\bar T$ is totally skew-symmetric on $\mathcal D=\ker(\eta)$,
\item[2)] the $(1,1)$-tensor field $\tau$ defined by
\[\tau X=\bar T(\xi, X)\]
satisfies
\begin{equation}\label{tau_phi}
\tau\phi+\phi\tau=-2(r+1)\phi^2.
\end{equation}
\end{itemize}
This linear connection is given by:
\begin{equation}\label{canonic0}
\bar\nabla_XY=\nabla_XY+H(X,Y)
\end{equation}
where
\begin{equation}\label{canonic}
H(X,Y)=\frac{1}{2}(\nabla_X\phi)\phi Y-r\,\eta(X)\phi Y+\eta(Y)(\phi-h)X-\frac{1}{2}g((\phi-h)X,Y)\xi.
\end{equation}
\end{theorem}
\begin{proof}
Let \ us \ consider \ the \ $(0,3)$-tensors \ defined \ by \ $H(X,Y,Z):=g(H(X,Y),Z)$ \ and \ $\bar T(X,Y,Z):=g(\bar T(X,Y),Z)$. First,
we prove that the linear connection defined by \eqref{canonic0} and \eqref{canonic} parallelizes the structure.
Notice that $H(X,\xi)=\phi X-hX=-\nabla_X\xi$, and thus $\bar\nabla_X\xi=0$. The linear connection is metric if and only if
\begin{equation}\label{H-metric}
H(X,Y,Z)+H(X,Z,Y)=0.
\end{equation}
We compute,
\begin{align}\label{nablaphi2}
(\nabla_X\phi)\phi Y+\phi (\nabla_X\phi)Y&=(\nabla_X\phi^2)Y\nonumber\\&=(\nabla_X\eta)(Y)\xi+\eta(Y)\nabla_X\xi\nonumber\\
&= g(Y,\nabla_X\xi)\xi+\eta(Y)\nabla_X\xi\nonumber\\
&={}-g(Y,\phi X-hX)\xi-\eta(Y)(\phi X-hX).
\end{align}
A straightforward computation using \eqref{canonic} and \eqref{nablaphi2} gives \eqref{H-metric}.
Moreover, $\bar \nabla$ satisfies $\bar\nabla\phi=0$ if and only if
\begin{equation}\label{Hphi_parallel}
(\nabla_X\phi)Y+H(X,\phi Y)-\phi H(X,Y)=0,
\end{equation}
which is proved again by a simple computation using \eqref{nablaphi2}.
The torsion of  $\bar\nabla$ is given by
\begin{align*}
\bar T(X,Y)&=H(X,Y)-H(Y,X)\\
&=\frac{1}{2}((\nabla_X\phi)\phi Y-(\nabla_Y\phi)\phi X)-g((\phi-h)X,Y)\xi\\
&\quad -(r+1)(\eta(X)\phi Y-\eta(Y)\phi X)+\eta(X)hY-\eta(Y)hX.
\end{align*}
Now, applying \eqref{main} and \eqref{H-metric} we get
\begin{align*}
(\nabla_Y\phi)\phi X&=-\phi (\nabla_Y\phi)X-g(X,\phi Y-hY)\xi-\eta(X)(\phi Y-hY)\\
&= \phi (\nabla_X\phi)Y+\eta(X)\phi Y+\eta(Y)\phi X+g(\phi X-hX, Y)\xi-\eta(X)(\phi Y-hY)\\
&=-(\nabla_X\phi)\phi Y+\eta(X)hY+\eta(Y)hX.
\end{align*}
Therefore,
\begin{align*}
\bar T(X,Y)
&= (\nabla_X\phi)\phi Y-(r+1)(\eta(X)\phi Y-\eta(Y)\phi X)\\
&\quad +\frac{1}{2}\eta(X)hY-\frac{3}{2}\eta(Y)hX-g((\phi-h)X,Y)\xi.
\end{align*}
In particular, for every $X,Y,Z\in\mathcal{D}$, applying \eqref{nablaphi2}, we have
\begin{align*}
\bar T(X,Y,Z)+\bar T(X,Z,Y)
&=g((\nabla_X\phi)\phi Y+\phi(\nabla_X\phi)Y,Z)=0
\end{align*}
which proves condition 1). Finally,
\[\tau=(\nabla_\xi\phi)\phi-(r+1)\phi+\frac{1}{2}h=\frac{3}{2}h-(r+1)\phi,\]
which implies \eqref{tau_phi}.

We prove the uniqueness of the connection. Suppose that $\bar\nabla$ is a linear connection parallelizing the structure and whose torsion satisfies 1) and 2). We determine the tensor $H$ defined by \eqref{canonic0}.
First we prove that for every $X,Y,Z\in \mathcal{D}$,
\begin{equation}\label{HD}
H(X,Y,Z)=\frac{1}{2}g((\nabla_X\phi)\phi Y,Z).
\end{equation}
Since $\bar\nabla$ is a metric connection with totally skew-symmetric torsion on $\mathcal D$, 
for every $X,Y,Z\in \mathcal D$ we have
\begin{align*}
\bar T(X,Y,Z)&= \bar T(X,Y,Z)-\bar T(Y,Z,X)+\bar T(Z,X,Y)\\
&= H(X,Y,Z)- H(Y,X,Z)-H(Y,Z,X)\\&\quad+H(Z,Y,X)+H(Z,X,Y)-H(X,Z,Y)\\
&= 2H(X,Y,Z),
\end{align*}
and thus the tensor $H$ is totally skew-symmetric on $\mathcal D$.
Being $\bar\nabla\phi=0$, \eqref{Hphi_parallel} holds. Hence
\begin{equation}\label{Hphi}
H(X,Y,\phi Z)+H(X,\phi Y,Z)= - g((\nabla_X\phi)Y,Z).
\end{equation}
Now, we take the cycling permutation sum of the above formula. By the skew-symmetry of $H$ and \eqref{main}, we get
\[2\mathop{\Large{\frak S}}_{XYZ}H(X,Y,\phi Z)=-3g((\nabla_X\phi)Y,Z).\]
Substituting $Y$ with $\phi Y$, we have
\begin{equation}\label{Hcyclic}
2H(X,\phi Y,\phi Z)+2H(\phi Y,Z,\phi X)-2H(Z, X,Y)=-3g((\nabla_X\phi)\phi Y,Z).
\end{equation}
Now, applying \eqref{Hphi} and \eqref{main},
\begin{align*}
H(X,\phi Y,\phi Z)+H(\phi Y,Z,\phi X)&=-H(\phi Y,X,\phi Z)-H(\phi Y,\phi X,Z)\\
&=g((\nabla_{\phi Y}\phi)X,Z)\\
&=-g((\nabla_X\phi)\phi Y,Z).
\end{align*}
Hence, substituting in \eqref{Hcyclic}, we get \eqref{HD}.

Now, being $\bar\nabla\xi=0$, for every vector field $X$, we have $H(X,\xi)=-\nabla_X\xi=\phi X-hX$.
Moreover, since $\bar\nabla$ is a metric connection, then $H(X,Y,\xi)=-H(X,\xi,Y)$.
Therefore, it remains to determine $H(\xi,X)$.
By $\bar\nabla\phi=0$, we have
\begin{equation*}
H(\xi,\phi X)-\phi H(\xi,X)=-(\nabla_\xi\phi)X=-\phi hX.
\end{equation*}
We compute
\begin{align*}
(\tau\phi-\phi\tau)X&=\bar T(\xi,\phi X)-\phi \bar T(\xi,X)\\
&=H(\xi,\phi X)-H(\phi X,\xi)-\phi H(\xi,X)+\phi H(X,\xi)\\
&= -\phi hX-(\phi^2X-h\phi X)+\phi(\phi X-hX)\\
&= 3h\phi X.
\end{align*}
Combining the above formula with condition 2), we obtain
\[2\tau\phi=3h\phi-2(r+1)\phi^2.\]
Now, being $\tau\xi=0$, we get
\[\tau=\frac{3}{2}h-(r+1)\phi.\]
It follows that
\[H(\xi,X)=\bar T(\xi,X)+H(X,\xi)=\frac{1}{2}hX-r\phi X.\]
This completes the proof that $H$ coincides with the tensor defined in \eqref{canonic}.
\end{proof}

\begin{remark}
\emph{Suppose that $(M,\phi,\xi,\eta,g)$ is a Sasakian manifold. Recall  that the covariant derivative of $\phi$ is given by
\[(\nabla_X\phi)Y=g(X,Y)\xi-\eta(Y)X\]
(see \cite[Theorem 6.3]{BLAIR}). Then the tensor $H$ in \eqref{canonic} becomes:}
\[H(X,Y) = g(X,\phi Y)\xi-r\,\eta(X)\phi Y+\eta(Y)\phi X.\]
\emph{It follows that $\bar\nabla$ coincides with the linear connection defined by Okumura in \cite{Okumura} (see also \cite{Tak}). In the case $r=-1$, this is the Tanaka-Webster connection (cf. \cite{Tanno}). In the case $r=1$, this is the unique linear connection on the Sasakian manifold $M$ parallelizing the structure and with totally skew-symmetric torsion defined in \cite{FrIv}.}
\end{remark}

\begin{proposition}
Let $(M,\phi,\xi,\eta,g)$ be a $5$-dimensional nearly Sasakian manifold. Let $\bar\nabla$ be the canonical connection defined in \eqref{canonic0} and \eqref{canonic}. Then the structure tensor $h$ is parallel with respect to $\bar\nabla$ if and only if $r=\frac{1}{2}$.
\end{proposition}
\begin{proof}
Using \eqref{canonic} and \eqref{nablaphi-completo}, we can compute
\begin{equation}\label{H5}
H(X,Y)=\frac{1}{2}\eta(X)hY-r\eta(X)\phi Y+\eta(Y)(\phi X-hX)-g(\phi X-hX,Y)\xi.
\end{equation}
Now, using the above formula and \eqref{nablah-completo}, a straightforward computation gives
\begin{equation*}
(\bar\nabla_Xh)Y=(\nabla_Xh)Y +H(X,hY)-hH(X,Y)=(1-2r)\eta(X)\phi hY
\end{equation*}
which proves our claim.
\end{proof}

\begin{remark}\label{remark_canonic}
\emph{The canonical connection corresponding to $r=\frac{1}{2}$ actually parallelizes the $SU(2)$-structure $\{(\phi_i,\xi,\eta,g)\}_{i\in\{1,2,3\}}$, or equivalently $(\eta,\omega_1,\omega_2,\omega_3)$, associated to the nearly Sasakian non-Sasakian structure. Furthermore the torsion of the canonical connection is given by}
\[\bar T(X,Y)=\frac{3}{2}\{\eta(Y)(\phi X-hX)-\eta(X)(\phi Y-hY)\}-2g(\phi X-hX,Y)\xi,\]
\emph{which turns out to satisfy $\bar\nabla\bar T=0$.}

\emph{Now, if we apply the deformation \eqref{Sasaki-E}, also the Sasaki-Einstein $SU(2)$-structure $(\tilde\eta,\tilde\omega_1,\tilde\omega_2,\tilde\omega_3)$ is parallel with respect to the canonical connection $\bar\nabla$. Furthermore, by \eqref{h-5dim} and \eqref{deformation}, we obtain}
\[H(X,Y)=\tilde g(X,\tilde \phi Y)\tilde \xi-\frac{1}{2}\,\tilde \eta(X)\tilde \phi Y+\tilde \eta(Y)\tilde \phi X.\]
\emph{Therefore, the canonical connection $\bar\nabla$ coincides with the Okumura connection associated to the Sasakian structure $(\tilde \phi,\tilde \xi,\tilde \eta,\tilde g)$ for $r=\frac{1}{2}$.}
\end{remark}

In general, for a Sasaki-Einstein $5$-manifold we have the following
\begin{proposition}
Let $M$ be a Sasaki-Einstein $5$-manifold with $SU(2)$-structure $(\eta,\omega_1,\omega_2,\omega_3)$. Then the Okumura connection corresponding to $r=\frac{1}{2}$ and associated to the Sasakian structure $(\phi_3,\xi,\eta, g)$ parallelizes the whole $SU(2)$-structure.
\end{proposition}
\begin{proof}
The Okumura connection corresponding to $r=\frac{1}{2}$ and associated to the Sasakian structure $( \phi_3, \xi, \eta, g)$ is given by
\[\bar\nabla_XY=\nabla_XY+H(X,Y),\]
where
\begin{equation}\label{ok}
H(X,Y) = g(X,\phi_3 Y)\xi-\frac{1}{2}\eta(X)\phi_3 Y+\eta(Y)\phi_3 X.
\end{equation}
By Corollary \ref{SE-nearly}, the almost contact metric structure $(\phi_2,\xi,\eta,g)$ is nearly cosymplectic, and thus
\[3g((\nabla_X\phi_2)Y,Z)=d\omega_2(X,Y,Z)=-3(\eta\wedge\omega_1)(X,Y,Z).\]
Therefore, an easy computation gives
\[(\nabla_X\phi_2)Y=g(X,\phi_1Y)\xi-\eta(X)\phi_1Y+\eta(Y)\phi_1X.\]
Using the above equation, \eqref{ok} and $\phi_2\phi_3=\phi_1=-\phi_3\phi_2$, we have
\[(\bar\nabla_X\phi_2)Y=(\nabla_X\phi_2)Y+H(X,\phi_2 Y)-\phi_2 H(X,Y)=0.\]
Hence, all the structure tensors $(\phi_i,\xi,\eta,g)$, $i\in\{1,2,3\}$, are parallel with respect to $\bar\nabla$.
\end{proof}

\end{document}